\documentclass[a4paper,11pt]{amsart}
\usepackage{amsmath, amsfonts, amssymb, amsthm, amscd}
\usepackage[colorlinks=true, linkcolor=blue, citecolor=blue]{hyperref}
\usepackage{graphicx}
\usepackage{enumerate}
\usepackage{url}
\usepackage{color}
\usepackage[utf8]{inputenc}
\usepackage{tikz}
\usetikzlibrary{arrows} 
\usepackage[a4paper]{geometry}
\usepackage{dsfont}
\usepackage{mathrsfs}

\numberwithin{equation}{section}

\newtheorem{theorem}{Theorem}[section]
\newtheorem{lemma}[theorem]{Lemma}
\newtheorem{proposition}[theorem]{Proposition}
\newtheorem{corollary}[theorem]{Corollary}
\newtheorem{remark}[theorem]{Remark}

\newcommand{\bbC}{{\ensuremath{\mathbb C}} }

\newcommand{\bbN}{{\ensuremath{\mathbb N}} }

\newcommand{\bbR}{{\ensuremath{\mathbb R}} }

\newcommand{\bbZ}{{\ensuremath{\mathbb Z}} }

\newcommand{\cA}{{\ensuremath{\mathcal A}} }
\newcommand{\cB}{{\ensuremath{\mathcal B}} }

\newcommand{\cD}{{\ensuremath{\mathcal D}} }
\newcommand{\cE}{{\ensuremath{\mathcal E}} }
\newcommand{\cF}{{\ensuremath{\mathcal F}} }
\newcommand{\cG}{{\ensuremath{\mathcal G}} }
\newcommand{\cH}{{\ensuremath{\mathcal H}} }

\newcommand{\cK}{{\ensuremath{\mathcal K}} }

\newcommand{\cM}{{\ensuremath{\mathcal M}} }
\newcommand{\cN}{{\ensuremath{\mathcal N}} }

\newcommand{\cR}{{\ensuremath{\mathcal R}} }
\newcommand{\cS}{{\ensuremath{\mathcal S}} }

\newcommand{\cU}{{\ensuremath{\mathcal U}} }

\newcommand{\cX}{{\ensuremath{\mathcal X}} }

\newcommand{\ga}{\alpha}
\newcommand{\gb}{\beta}

\newcommand{\gd}{\delta}
\newcommand{\gD}{\Delta}
\newcommand{\gep}{\varepsilon} 

\newcommand{\gt}{\theta}
\newcommand{\gk}{\kappa}
\newcommand{\gl}{\lambda}
\newcommand{\gL}{\Lambda}

\newcommand{\gs}{\sigma}
\newcommand{\gS}{\Sigma}

\newcommand{\bfD}{\mathbf{D}}

\renewcommand{\tilde}{\widetilde}          
\DeclareMathSymbol{\leqslant}{\mathalpha}{AMSa}{"36} 
\DeclareMathSymbol{\geqslant}{\mathalpha}{AMSa}{"3E} 
\DeclareMathSymbol{\eset}{\mathalpha}{AMSb}{"3F}     
\newcommand{\dd}{\text{\rm d}}             
\newcommand{\suptwo}[2]{\sup_{\substack{#1 \\ #2}}} 
\newcommand{\inftwo}[2]{\inf_{\substack{#1 \\ #2}}} 
\newcommand{\infthree}[3]{\inf_{\substack{#1 \\ #2\\ #3\\}}} 
\newcommand{\sumtwo}[2]{\sum_{\substack{#1 \\ #2}}} 

\newcommand{\R}{\mathbb{R}}

\newcommand{\N}{\mathbb{N}}

\newcommand\bP{\ensuremath{\mathrm{P}}}
\newcommand\bE{\ensuremath{\mathrm{E}}}

\renewcommand{\epsilon}{\varepsilon}

\newcommand{\card}{\mathrm{card}}

\newcommand{\cp}{\mathrm{cap}}

\newcommand{\bra}{\langle}
\newcommand{\ket}{\rangle}

\newenvironment{myenumerate}{
\renewcommand{\theenumi}{\arabic{enumi}}
\renewcommand{\labelenumi}{{\rm(\theenumi)}}
\begin{list}{\labelenumi}
{
\setlength{\itemsep}{0.4em}
\setlength{\topsep}{0.5em}
\setlength\leftmargin{2.45em}
\setlength\labelwidth{2.05em}
\setlength{\labelsep}{0.4em}
\usecounter{enumi}
}
}
{\end{list}
}

\newcommand{\beq}{\begin{equation}}
\newcommand{\eeq}{\end{equation}}
\newcommand{\ba}{\begin{aligned}}
\newcommand{\ea}{\end{aligned}}

\newcommand{\fm}{\mathfrak{m}}
\newcommand{\lra}{\leftrightarrow}

\newcommand{\bin}{\mathrm{Bin}}

\newcommand{\leb}{\mathrm{Leb}}
\newcommand{\tv}{\mathrm{tv}}
\newcommand{\loc}{\mathrm{loc}}

\begin{document}

\title[]{Uniqueness and tube property for the Swiss cheese large deviations}
\author{Dirk Erhard and Julien Poisat}

\begin{abstract}
We consider the simple random walk on the Euclidean lattice, in three dimensions and higher, conditioned to visit fewer sites than expected, when the deviation from the mean scales like the mean. The associated large deviation principle was first derived by van den Berg, Bolthausen and den Hollander (Annals of Mathematics, 2001) in the continuous setting, that is for the volume of a Wiener sausage, and later taken up by Phetpradap in the discrete setting. One of the key ideas in their work is to condition the range of the random walk on a certain skeleton, that is a sub-sequence of the random walk path along an appropriate intermediate scale. In this paper we answer two questions that were left open. First, we prove using analytical tools for quasilinear elliptic equations that the rate function obtained by van den Berg, Bolthausen and den Hollander has a unique minimizer over the set of probability measures modulo shifts, for almost every deviation in dimension three and for deviations below a certain threshold in dimensions four and higher. Secondly, we adapt a compactification method for the space of probability measures that was recently introduced by Mukherjee and Varadhan in order to (i) state a {\it strong} large deviation principle for the skeleton occupation measure under the small range constraint and (ii) prove that the latter converges, under the conditioned law, to the minimizer of the rate function modulo shifts. 

\end{abstract}
 
\date{\today}

\maketitle

\tableofcontents

\section*{Acknowledgements}
The authors are much grateful to two anonymous referees for their careful reading and for pointing out several issues in a preliminary version of this paper as well as in the companion paper~\cite{EP23}, in particular about the lack of continuity of the projection map $\tilde\Pr$ and the closedness of the set $\cM(A,\gep_0)$, both defined in Section~\ref{sec:ldp}. We explain how those issues were tackled below Proposition~\ref{pr:LDPpair} and Lemma~\ref{lem:Projection}, respectively.\\

\par We would also like to thank Erwin Bolthausen, Frank den Hollander, Jimmy Lamboley and Chiranjib Mukherjee for stimulating discussions during the preparation of this manuscript.\\

\par JP acknowledges the support of ANR LOCAL (ANR-22-CE40-0012) and the hospitality of UFBA (Universidade Federal da Bahia). D.E. was supported by the National Council for Scientific and Technological Development - CNPq via a Bolsa de Produtividade 318182/2025-4 and via a Universal Grant (Grant Number 401314/2025-1). D.E.	moreover acknowledges support by the Serrapilheira Institute (Grant Number Serra-R-2011-37582). D.E. was partially supported by FAPESB (EDITAL FAPESB Nº 012/2022 - UNIVERSAL - NºAPP0044/2023). FAPESB is the Bahia Research Foundation. D.E. moreover acknowledges the hospitality of the University of Paris Dauphine.

\section{Introduction}
\par  Let $S=(S_n)_{n\in\bbN_0}$ be a discrete-time simple random walk on $\bbZ^d$, whose increments $(S_n - S_{n-1})_{n\in \bbN}$ are independent and uniformly distributed on the $2d$ unit vectors. We assume throughout the paper that $d\ge 3$ and we denote by $\bP_x$ and $\bE_x$ the probability and expectation with respect to the simple random walk starting from $x\in \bbZ^d$. We omit the subscript when $x=0$. The {\it range} of the random walk up to time $n$ is the set of all vertices visited by the process up to time $n$, which we denote by
\beq
S_{(0,n]} = \{S_1, \ldots, S_n\},
\eeq
and its cardinality (volume of the range) is denoted by
\beq
R_n = \#\{S_1, \ldots, S_n\}.
\eeq
The almost-sure asymptotic behavior for the volume of the range is given by the following Law of Large Numbers~\cite{DE51}
\beq
\label{eq:LLN-range}
\frac{R_n}{n} \longrightarrow \gk_d := 
\bP(S_k \neq 0,\ \forall k\ge 1),
\qquad n\to \infty.
\eeq
The limit, sometimes called the {\it escape probability}, is positive as the random walk is transient in dimension $d\ge 3$. This probability coincides with the inverse of the Green function at the origin:
\beq
1/\gk_d = \sum_{n\ge 0} \bP(S_n = 0).
\eeq
In this paper we consider the law of the simple random walk conditioned on the large deviation event that $R_n \le bn$, in the limit of large $n$ and when $b< \gk_d$, that is for {\it lower} deviations {\it scaling like the mean}.
\subsection{Moderate deviations of the range of simple random walk}
\par The problem of evaluating the probability of the large deviation event above was first considered by van den Berg, Bolthausen and den Hollander~\cite{BBH2001} in the context of Wiener sausages instead of random walks. Let $B = (B_s)_{s\ge 0}$ be a standard Brownian motion on $\bbR^d$ and $a>0$. The Wiener sausage with radius $a$ and running up to time $t\ge 0$ is defined as the random subset
\beq
W^a(t) = \{x\in \bbR^d \colon \exists s\in[0,t] \colon |x - W_s| \le a\}, 
\eeq
where throughout the paper $|\cdot|$ denotes the Euclidean norm. The following almost-sure convergence~\cite{S64,WhitmanPhD64} is the analogue of~\eqref{eq:LLN-range} in the continuum:
\beq
\label{eq:def-gka}
\frac{\leb [W^a(t)]}{t} \longrightarrow \gk_a := \cp(\bar \cB(0,a)), \qquad t\to \infty,
\eeq
where $\leb$ is Lebesgue measure, $\bar \cB(0,a)$ is the closed Euclidean ball with radius $a$ centred at the origin, and $\cp(\bar \cB(0,a))$ denotes its Newtonian capacity.
\par For any $\gk>0$ and $b>0$, define 
\beq
\label{def:Ifunction}
I_{\gk}(b) = \inf_{\phi \in \cD_\gk(b)}
\Big[
\frac12 \int_{\bbR^d} |\nabla \phi (x)|^2 \dd x
\Big],
\eeq
where
\beq
\cD_\gk(b) = 
\Big\{
\phi \in H^1(\bbR^d) \colon
\int_{\bbR^d} \phi^2(x)\dd x = 1,
\int_{\bbR^d} (1-e^{-\gk \phi^2(x)})\dd x \le b
\Big\}.
\eeq
The main result in~\cite{BBH2001} reads:
\begin{theorem}[Van den Berg, Bolthausen and den Hollander : moderate deviations for the volume of the Wiener sausage] Assume $d\ge 3$. For all $b>0$,
\label{thm:BBH2001}
\beq
\label{eq:BBH2001}
\lim_{t\to\infty} t^{\frac{2}{d}-1}\log \bP(\leb[W^a(t)] \le bt) = -I_{\gk_a}(b),
\eeq
where $\gk_a$ is chosen as in~\eqref{eq:def-gka}. 
\end{theorem}
Here, the term {\it moderate} refers to the exponent $1-2/d$ being smaller than one.
\begin{remark}
It was shown in~\cite[Theorem 3]{BBH2001} that $I_{\gk_a}(b)>0$ if and only if $b\in (0,\gk_a)$. 
\end{remark}
 This result was later adapted to the random walk setting in Phetdradap's Ph.D thesis~\cite{Phetpradap}.
\begin{theorem}[Phetdradap : deviations for the range of the simple random walk] \label{thm:MDP-rw}
Assume $d\ge 3$. For all $b>0$,
\beq\label{eq:MDP-rw}
\lim_{n\to\infty} n^{\frac{2}{d}- 1} \log \bP(R_n \le bn) = - \tfrac1d I_{\gk_d}(b),
\eeq
where $\gk_d$ is chosen as in~\eqref{eq:LLN-range}.
\end{theorem}
The function that governs the large deviation is the same in both cases, up to a multiplicative constant $1/d$. The lattice structure survives in the limit through the constant $\gk_d$ (instead of $\gk_a$). It may be however pulled out from the variational formula by a simple scaling argument. In the sequel we shall write $\gk$ instead of $\gk_d$ and drop the dependence of $\cD(b)$ and $I(b)$ on $\kappa$, in order to lighten notation. This should not lead to any confusion since we only deal with the simple random walk from now on.

\par While Theorems~\ref{thm:BBH2001} and~\ref{thm:MDP-rw} settle the issue of the large deviation cost, the question about the law of the random walk {\it conditioned} on this large deviation event remains. In the continuous setting, van den Berg, Bolthausen and den Hollander~\cite{BBH2001} set forth a heuristic picture coined as the {\it Swiss cheese} strategy: the conditioned Brownian motion should behave as if pushed by a drift field towards the origin, folding itself onto scale $t^{1/d}$ instead of the typical scale $t^{1/2}$. While doing so, the Wiener sausage \emph{covers only part of the space and leaves random holes whose sizes are of order one and whose density varies on scale $t^{1/d}$}, to quote the authors.
The function that drives the drift field is expected to be the minimizer of the rate function, provided there exists a unique minimizer, at least modulo spatial shifts.

The key idea in the works of van den Berg, Bolthausen and den Hollander \cite{BBH2001} and in Phetdradap~\cite{Phetpradap} is that it is possible to approximate the volume of the Wiener sausage or the range of the random walk by a certain functional of the occupation time measure of another discretised random walk, which we will refer to as the skeleton occupation time measure. The proofs of Theorem~\ref{thm:BBH2001} and Theorem~\ref{thm:MDP-rw} essentially establish a large deviation principles for the skeleton occupation time measure. It is therefore only natural to conjecture that the best strategy for the random walk to realise the large deviation event in~\eqref{eq:MDP-rw} is that the skeleton occupation measure mimics the profile of the minimizer of the variational problem on the right-hand side of~\eqref{eq:MDP-rw}. However, a significant challenge arises in that the issue of the uniqueness of this minimizer had been left unanswered since~\cite{BBH2001}. In this work, we tackle this question by employing tools for quasilinear elliptic equations. This matter is discussed in Section~\ref{sec:exi-uni-min} below.
Provided existence and uniqueness, we  show that the minimizer is the limiting profile for the skeleton occupation time measure, thus confirming the above conjecture. The limit is for a certain topology, explained in Section~\ref{sec:compact-topo0}, that disregards space shifts. This property, referred to as {\it tube property}, is exposed in Section~\ref{sec:tube-pro}. We believe that the two main results of this paper (Theorems~\ref{thm:1} and~\ref{thm:2}) are a first step towards a rigorous description of that so-called {\it Swiss cheese} picture, conjectured to be linked to the model of random interlacements~\cite{sznitman2021bulk}. To the authors knowledge, the only available results on the conditioned random walk path were obtained by Asselah and Schapira~\cite{AssSch17,Asselah:2018aa, MR4265025}.

\subsection{Existence and uniqueness of minimizers modulo shifts}
\label{sec:exi-uni-min}
We first need to recall some definitions from~\cite{BBH2001}. Let $D^1(\bbR^d)$ be the set of locally integrable functions $\psi\colon \bbR^d\mapsto \bbC$ such that $\nabla \psi \in L^2(\bbR^d)$ (in the sense of distributions) and such that for all $a>0$ the set $\{|\psi|>a\}$ has finite Lebesgue measure~\cite{LiebLoss}. When $d\ge 5$, let us define
\beq
\label{def:u*d}
u_d^* := 1- \Big[\inf_{\psi\in\gS^*}\|\psi\|_2^2\Big]^{-1}
\eeq
where $\|\cdot\|_2$ is the $L^2$-norm and $\gS^*$ is the set of local minimizers of $\|\nabla\psi\|_2$ among all $\psi\in D^1(\bbR^d)$ such that $\int_{\bbR^d} (e^{-\psi^2}-1+\psi^2) = 1$. It is known that $2/d \le u_d^* < 1$~\cite[Theorem 5]{BBH2001}. In~\cite[Theorems~4 and 5]{BBH2001}, the authors proved that, for all $b\in(0,\gk)$ when $d\in\{3,4\}$ or for all $b\in(0,u_d^*\gk]$ when $d\ge 5$, the variational problem in~\eqref{def:Ifunction} has a minimizer that is strictly positive, has a unique global maximum, is radially symmetric (modulo shifts) and strictly decreasing in the radial component. Moreover, any other minimizer is of the same type. However, uniqueness was still open. In this paper we prove the following:

\begin{theorem}[Uniqueness of minimizers]
	\label{thm:1}
	The variational problem in~\eqref{def:Ifunction} has at most one minimizer (modulo space shifts) for almost every $b\in (0, \gk)$ when $d=3$ and almost every $b\in (0, \frac2 d \gk)$ when $d\ge 4$.
\end{theorem}
\begin{remark}
If we were to know that $b\mapsto I(b)$ were differentiable, then our proof would allow us to remove the ``almost every" part from our statement. Our method to prove uniqueness fails when $b$ is close enough to $\gk$ and $d\ge 4$. It thus remains an open question to determine whether there is a unique minimizer modulo spatial shifts for all $b\in(0,\gk)$ and $d=4$. When $b$ is close enough to $\gk$ and $d\ge 5$, it is known that there is no minimizer but rather a minimizing sequence of probability measure that loses mass~\cite[Theorem 5]{BBH2001}. As we shall see below, one can embed the space of sub-probability measures modulo shifts into a larger space (its compactification with respect to a certain topology) on which we may also write a variational principle. The fundamental question then is if this new variational problem characterizes the Swiss cheese in the sense that one has equality in~\eqref{eq:MDP-rw} when one replaces the right hand side in~\eqref{eq:MDP-rw} by the new variational problem and whether there exists a unique minimizing \emph{sub}-probability measure modulo spatial shifts. 
	\end{remark}
From what precedes, we may assert that the variational problem in~\eqref{def:Ifunction} has a unique minimizer (modulo space shifts) for almost all $b\in (0, \gk)$ when $d=3$, and for (at least) almost all $b\in(0,\tfrac2 d\gk)$ when $d\ge 4$. For such values of $b$, we shall then denote by $\phi_b$ the unique minimizer centered at the origin, and by
\beq
\fm_b := \{\phi_b^2 * \gd_x \colon x\in \bbR^d\}
\eeq
the set of minimizers, where $*$ is the convolution operation, $\gd_x$ is the Dirac mass at $x\in\bbR^d$ and, with a slight abuse of notation, $\phi_b^2$ stands for the measure with density $\phi_b^2$ with respect to Lebesgue measure. 
As we shall see, $\fm_b$ is the limit of the occupation time measure of a certain {\it skeleton} of the random walk conditioned on the large deviation event. In the next section we introduce the topology under consideration.
\subsection{Compactification of the space of probability measures}
\label{sec:compact-topo0}
The empirical and pair empirical measures of many Markov chains and processes such as simple random walk on $\bbZ^d$ or Brownian motion on $\bbR^d$ only satisfy a weak large deviation principle. This is due to the lack of exponential tightness. However, the fact that the large deviation upper bound only holds for compact sets is often a big obstacle. In our context this is not different. To circumvent that problem Mukherjee and Varadhan~\cite{MV2016} introduced a new topology which takes the shift invariance of many models in statistical mechanics into account and allows to compactify the space of measures, see also \cite{BatCha20,BKM2017,BroMuk19,KM2017} for applications and the end of this section for a brief historical account. In this section we summarize the construction of this topology.
\par Let $\cM_1= \cM_1(\bbR^d)$ be the space of probability measures on $\bbR^d$ and $\cM_{\le1} = \cM_{\le1}(\bbR^d)$ be the space of sub-probability measures on $\bbR^d$. We consider the action of the shifts $\gt_x$, for $x\in\bbR^d$, defined by:
\beq
\int_{\bbR^d} f(u) (\gt_x \ga)(\dd u) = \int_{\bbR^d} f(u+x)  \ga(\dd u)
\eeq
for all continuous and bounded functions $f\colon \bbR^d\mapsto \bbR$ and $\ga\in \cM_{\le1}$. We shall denote by $\tilde \cM_1$ (resp. $\tilde \cM_{\le 1}$) the space of equivalence classes of $\cM_1$ (resp. $\cM_{\le 1}$) under the action of the shifts $\gt_x$. For any $\alpha\in \cM_{\le 1}$ we denote by $\tilde\alpha$ its orbit, i.e., equivalence class. For $k\ge 2$, we define $\cF_k$ as the space of continuous functions $f\colon (\bbR^d)^k \mapsto \bbR$ that are {\it translation invariant}, i.e.
\beq
f(u_1+x, \ldots, u_k +x) = f(u_1,\ldots, u_k), \quad \forall x,u_1,\ldots, u_k \in \bbR^d,
\eeq
and {\it vanishing at infinity}, in the sense that
\beq
\lim_{\max_{i\neq j} |u_i-u_j|\to \infty} f(u_1, \ldots, u_k) = 0.
\eeq
For $k\ge 2$, $f\in \cF_k$ and $\ga \in \cM_{\le 1}$, we write
\beq
\label{eq:defLambda0}
\gL(f, \ga) := \int f(u_1,\ldots, u_{k}) \prod_{1\le i\le k} \ga(\dd u_i),
\eeq
which actually only depends on the orbit $\tilde \ga$. 
Let us define
\beq
\cF := \bigcup_{k\ge 2} \cF_k,
\eeq
for which there exists a countable dense set (under the uniform metric) denoted by
\beq
\{f_r(u_1,\ldots, u_{k_r}),\, r\in \bbN\}.
\eeq
We then define the following set of empty, finite, or countable collections of sub-probability measure orbits:
\beq
\tilde\cX := \Big\{
\xi = \{\tilde \ga_i\}_{i\in I} \colon \tilde \ga_i \in \tilde \cM_{\le 1},\ \sum_{i\in I} \ga_i(\bbR^d) \le 1\Big\}.
\eeq
For every $\xi_1, \xi_2\in \tilde \cX$, define
\beq
\mathbf{D}(\xi_1, \xi_2) := \sum_{r\ge 1} \frac{1}{2^r} \frac{1}{1+\|f_r\|_{\infty}}
\Big|%
\sum_{\tilde\ga\in\xi_1} \gL(f_r, \ga)- \sum_{\tilde\ga\in\xi_2} \gL(f_r, \ga)
\Big|.%
\eeq
It was then shown in~\cite{MV2016} that the space $\tilde\cX$ equipped with $\mathbf{D}$ is a compact metric space and that $\tilde \cM_1$ is dense in $\tilde\cX$. Moreover, the set of sub-probability measure orbits is naturally embedded into $\tilde\cX$. We refer the interested reader to~\cite{MV2016} for details. Let us however close this section with a simple (one-dimensional) example in order to better grasp the idea behind this topology. Consider a sequence of probability measure $(\mu_n)_{n\ge 1}$ defined by
\beq
\mu_n = \frac12 \cN(n,1) + \frac13 \cN(-n,2) + \frac16 \cN(0,n),
\eeq
where $\cN(m,\gs^2)$ is the normal distribution with mean $m$ and variance $\gs^2$. This sequence does not converge in the weak topology but the sequence $(\tilde \mu_n)_{n\ge 1}$ does converge in the $\bfD$-topology to the limit $\xi=\{\frac12 \cN(\cdot, 1), \frac13 \cN(\cdot, 2)\}\in\tilde\cX$, where $\cN(\cdot, \gs^2)$ denotes the normal distribution modulo space shifts. The reason behind that is that the two components $\tfrac12 \cN(n,1)$ and $ \tfrac13 \cN(-n,2)$ coincide after a shift with $\tfrac12 \cN(0,1)$ and $ \tfrac13 \cN(0,2)$, whereas the last component of $\mu_n$ simply goes to zero.

\subsection{Tube property}
\label{sec:tube-pro} 
In this section we state the second and last main theorem of this paper.
Let $\gep>0$ and $n\in \bbN$. We cut the random walk trajectory in blocks of length
\beq
\label{eq:mesoscopic-scale}
\ell := \ell(n,\gep) = \lfloor \gep n ^{2/d} \rfloor.
\eeq
The number of blocks is denoted by (we assume that $n\in \ell \bbN$ for simplicity)
\beq\label{eq:M}
M = \frac n \ell \sim \frac1\gep n^{1-2/d},
\eeq
as $n\to \infty$. The (renormalized) {\it skeleton} process is defined as
\beq
\hat S_i^{(\gep)} := \frac{S_{i\ell}}{n^{1/d}} \in \frac{\bbZ^d}{n^{1/d}},
\qquad 0\le i \le M,
\eeq
and its pair empirical measure, which is a random measure on $(\frac{\bbZ^d}{n^{1/d}})^2$, is denoted by
\beq
\label{eq:def-PEMa}
L_{M,\gep}^{(2)} :=\frac{1}{M} \sum_{0< i \le M} \gd_{(\hat S_{i-1}^{(\gep)}, \hat S_{i}^{(\gep)})}\,,
\eeq
where $\gd$ is the Dirac measure. We might sometimes omit the subscript $\gep$ to lighten notations. 
Although the pair empirical measure is central in the proof, only its first marginal is necessary to state our second main result. We denote it by
\beq
L_{M,\gep} :=\frac{1}{M} \sum_{0\le i < M} \gd_{\hat S_{i}^{(\gep)}}\,.
\eeq
%
\begin{theorem}[Tube property]
\label{thm:2}
Let $b\in(0,\kappa)$ if $d=3$ or $b\in(0,\tfrac 2 d \kappa)$ if $d\ge 4$ be such that the variational problem in~\eqref{def:Ifunction} has a unique minimizer modulo space shifts. Let $\cU(\fm_b)$ be an open neighbourhood of $\fm_b$ w.r.t.\ the $\bfD$-topology. There exists $\gep_0$ such that for $\gep\in(0,\gep_0)$,
\beq
\limsup_{n\to \infty} n^{\frac{2}{d}-1} \log \bP(\tilde L_{M,\gep} \notin \cU(\fm_b) | R_n \le bn) <0.
\eeq
\end{theorem}
\begin{remark}
It is natural to wonder if the above theorem also holds in the original setting considered in~\cite{BBH2001} where the random walk is replaced by a Wiener sausage. We expect this to be the case, and the approach adopted here should also work in that setting. However, we did not verify all the details so we refrain from stating it as a theorem.
\end{remark}
\begin{remark}[Case $d=2$] Although the focus of this paper is on dimensions three and higher, let us write a few words about the two-dimensional case. It was proven in~\cite[Theorem 2]{BBH2001} that for $d=2$,~\eqref{eq:BBH2001} becomes
\beq
\label{eq:d=2}
\lim_{t\to\infty} \frac 1 {\log t}\log \bP\Big(\leb[W^a(t)] \le \frac{bt}{\log t}\Big) = -I_{2\pi}(b),
\eeq
for every $a,b>0$. This was later brought to the discrete setting in~\cite{Liu-Vogel21}. The authors in~\cite{BBH2001} claim that the optimal strategy to achieve the large deviation event is the same as in other dimensions. Nevertheless, they point out a few differences. For instance, the rate function has a finite slope at $b=2\pi$ in $d=2$ whereas it has an infinite slope at $b=\kappa_a$ in $d\ge 3$. Also, they note that the r.h.s.\ in~\eqref{eq:d=2} does not depend on $a$ (the radius of the sausage), which they interpret as the fact that the random holes in the Swiss cheese have a typical size and a typical mutual distance that tend to infinity as $t\to\infty$. Although we expect that a tube property holds in $d=2$ as well, let us stress that one of the key results we use in Section~\ref{sec:uni} (see Proposition~\ref{pr:serrin-tang}) is restricted to $d\ge 3$.
\end{remark}
Theorem~\ref{thm:2} relies on a strengthening of the large deviation upper bound in Theorem~\ref{thm:MDP-rw} to the $\bfD$-topology introduced in Section~\ref{sec:compact-topo0}. To ease notation we define for $\ga\in\cM_{\le 1}(\bbR^d)$
\beq
J(\ga) :=
\begin{cases}
J(\phi) = \frac12 \|\nabla\phi\|_2^2 & \textrm{ if } \phi:= \sqrt{\frac{\dd \ga}{\dd x}} \text{ exists and is in } H^1(\bbR^d),\\
+\infty & \textrm{ else, }
\end{cases}
\eeq
\beq
\Gamma(\ga) := 
\begin{cases}
\Gamma(\phi) = \int (1-e^{-\gk \phi^2(x)})\dd x & \text{if } \phi:= \sqrt{\frac{\dd \ga}{\dd x}} \text{ exists, }\\
+\infty & \text{else.} 
\end{cases}
\eeq
Both functions are translation invariant and may be extended to $\tilde \cX$ by setting:
\beq
\label{eq:def-tilde-JG}
\tilde J(\xi) := \sum_{i\in I} J(\ga_i), \qquad
\tilde\Gamma(\xi) := \sum_{i\in I} \Gamma(\ga_i) + \kappa\Big(1- \sum_{i\in I} \ga_i(\bbR^d) \Big),
\qquad \xi =\{\tilde\ga_i\}_{i\in I}.
\eeq
We draw the reader's attention to the fact that even though $\tilde \cM_{\le 1}$ is embedded into $\tilde \cX$, $\tilde \Gamma(\{\tilde \ga\}) = \Gamma(\ga)$ if and only if $\ga$ is a {\it probability} measure. We will comment on the presence of the second term in the definition of $\tilde \Gamma$ further in the paper, see Lemma~\ref{lem:scaling} and the comment just below.
Theorem~\ref{thm:2} then relies on the following extension of Theorem~\ref{thm:MDP-rw}.
\begin{proposition}[Large Deviations Upper Bound at the level of orbits]
\label{pr:LDP-measure}
Assume $d\ge 3$.
For any set $F\subseteq\tilde\cX$ closed in the $\bfD$-topology,
\beq
\label{eq:ldp-aux}
\limsup_{\gep\to 0} \limsup_{n\to\infty} n^{\frac2d -1}\log \bP(R_n\le bn, \tilde L_{M,\gep} \in F)
\le - { \tfrac 1d}\inf_{\xi \in F \cap \tilde\cD(b)} \tilde J(\xi),
\eeq
where
\beq
\tilde\cD(b) = \Big\{\xi\in\tilde\cX \colon \tilde\Gamma(\xi) \le b\Big\}.
\eeq
\end{proposition}
 Note that only the elements of $\tilde\cX$ with finite entropy contribute to the infimum in~\eqref{eq:ldp-aux} so the value assigned to $\Gamma$ for sub-probability measures with no density w.r.t.\ Lebesgue measure is actually irrelevant.\\
 
Let us close this section with a short historical account in order to better contextualize our results. The problem of large deviations for Wiener sausages (or the range of random walks in the discrete setting) has a rich and extensive history, dating back at least to the influential work of Donsker and Varadhan~\cite{DV79}. This has sparked considerable follow-up research, including work on random walks in random obstacles~\cite{Sz98} and Swiss cheese large deviations~\cite{Phetpradap, BBH2001,BBH2004}. Extracting the actual path behaviour from the asymptotics of the partition function in statistical mechanics models is often a delicate task~\cite{BC21,B94,MR1434118,DFSX18,ding2019distribution,DIngXu2019,DingXu2019LOC,Pov99} and no general method was known until recently. The compactification method recently developed in~\cite{MV2016}, which has a general scope, has already been successfully applied to several models exhibiting shift-invariance, including the mean-field Polaron measure and the construction of the Pekar measure~\cite{BKM2017,KM2017,MV2016}, construction of the (Fröhlich) Polaron measure~\cite{MC2020b}, its convergence towards  the Pekar process in
strong coupling~\cite{MC2020} and most recently the quartic divergence rate of the effective mass in the Fröhlich
polaron~\cite{BMSV25}. Let us also mention applications to directed polymers and Gaussian multiplicative chaos~\cite{BatCha20,BroMuk19}. Our main contribution in this paper is to bring this novel idea to the Swiss cheese large deviations. We expect that the techniques used in~\cite{BKM2017,KM2017} could be applied to our context in order to obtain further results, like tightness and distribution for the center of mass, that is finding the minimizer (among all the possible shifted versions) that actually best fits the empirical measure in the limit of a large system. However it seems that the discrete setting and the additional parameter $\gep$ coming from the skeleton decomposition in~\eqref{eq:mesoscopic-scale} lead to substantial technical complications, so we leave this task for future work.\\

\par The rest of the paper is organized as follows. Section~\ref{sec:uni} is devoted to the proof of Theorem~\ref{thm:1}, while Section~\ref{sec:tube} contains the proof of Theorem~\ref{thm:2} given the validity of Proposition~\ref{pr:LDP-measure}. The remaining part, that is Section~\ref{sec:ldp}, is dedicated to the proof of that proposition.
\section{Proof of Theorem~\ref{thm:1}: Uniqueness}
\label{sec:uni}
Let us first collect some known facts.
It was shown in \cite{BBH2001} that any minimizer of~\eqref{def:Ifunction} satisfies the Euler-Lagrange equation 
\beq\label{eq:EL}
\gD \phi + f_{\gl,\mu}(\phi) = 0,
\eeq
where the nonlinearity $f_{\gl,\mu}$ is given by 
\begin{equation}
f_{\gl,\mu}(x)= \gl x + \mu x \gk e^{-\gk x^2}, \qquad x\in\bbR,
\end{equation}
$\gl$ and $\mu$ being the Lagrange multipliers. We further know from \cite{BBH2001} that (i) the set of minimizers is stable under space shifts and (ii) any minimizer of this variational problem is (up to spatial shifts) radially symmetric (strictly) decreasing with $\lim_{|x|\to\infty}\phi(x)= 0$. Hence, we may and will assume from now on that $\phi$ is maximized at the origin. This implies that we can assume that~\eqref{eq:EL} is equipped with the boundary conditions
\begin{equation}
\begin{cases}
\text{$\phi$ radially symmetric decreasing with }\lim_{|x|\to\infty}\phi(x)= 0,\\
\nabla \phi(0) = 0.
\end{cases}
\end{equation} 
Moreover, letting $y(r) := \phi(r, 0, \ldots, 0)$ for $r\ge 0$, the function $y$ satisfies the (one-dimensional) Euler-Lagrange equation (see~\cite[Proof of Lemma 11]{BBH2001})
\beq
\label{eq:EL1d-}
\begin{cases}
	y''(r) + \frac{d-1}{r} y'(r) +f_{\gl,\mu}(y(r))=0,\\
	\lim_{r\to\infty} y(r)=0,\\
	y'(0)=0.
\end{cases}
\eeq
The proof of Theorem~\ref{thm:1} consists of two steps:
\begin{itemize}
	\item (Step 1) We prove that the pair of Lagrange multipliers $(\gl,\mu)$ is uniquely determined by $b$.
	\item (Step 2) We prove that~\eqref{eq:EL1d-} has a unique solution. 
\end{itemize}
\textbf{Step 1.} \emph{Determination of the Lagrange multipliers.} In this step we prove the following
\begin{proposition}\label{pr:uniq-Lagr-mult}
Assume that $I$ is differentiable at $b\in(0,\gk)$. Let $\phi$ be a minimizer of the variational problem in~\eqref{def:Ifunction}. Then $\phi$ is a solution of the Euler-Lagrange equation in~\eqref{eq:EL} and the Lagrange multipliers satisfy:
	\beq
	\ba
	{\rm (i)}\quad & \gl + \mu b = 2(1-\tfrac 2 d) I(b),\\
	{\rm (ii)}\quad & \mu =2 I'(b).
	\ea
	\eeq
\end{proposition}
Note that Proposition~\ref{pr:uniq-Lagr-mult} uniquely determines the Lagrange multipliers $(\mu(b),\lambda(b))$.
\begin{proof}[Proof of Proposition~\ref{pr:uniq-Lagr-mult}]
	For the proof we fix a solution $\phi$ to~\eqref{eq:EL}. Let us first prove (i). By Pohozaev's identity~\cite[Proposition 1, Chapter 2]{BL83} $\phi$ satisfies
	\begin{equation}
	\|\nabla \phi\|_2^2 = \frac{2d}{d-2}\int F_{\gl,\mu}(\phi(x))\, dx,
	\end{equation}
	where
	\begin{equation}
	F_{\gl,\mu}(z)= \int_0^z f_{\gl,\mu}(v)\, \dd v 
	= \frac12 \gl z^2 + \frac12 \mu (1-e^{-\gk z^2}).
	\end{equation}
	To apply this identity, one may actually check that
	\begin{itemize}
		\item $f_{\gl,\mu}\colon \bbR\to\bbR$ is continuous;
		\item $\phi\in L^\infty_{\loc}$;
		\item $\nabla \phi \in L^2$;
		\item $F_{\gl,\mu}(\phi)\in L^1$.
	\end{itemize}
	Thus, for any function $\phi$ that is a minimizer of~\eqref{def:Ifunction} and a solution to~\eqref{eq:EL}, we can write
	\begin{equation}
	I(b)= \tfrac12\|\nabla \phi\|_2^2= \frac{d}{2 (d-2)}\Big(\gl \|\phi\|_2^2 + \mu \int_{\R^d}\big(1-e^{-\gk\phi^2(x)}\big) \dd x\Big)=
	\frac{d}{2 (d-2)}(\gl +\mu b).
	\end{equation}
	Here, we used the fact that the minimizer is in $\cD(b)$ and that the second constraint is saturated as a consequence of~\cite[Lemma 12]{BBH2001} and its proof. We can deduce therefore that
	\begin{equation}\label{eq:Phozaev}
	\lambda +\mu b= \frac{2 (d-2)}{d}I(b).
	\end{equation}
	Equation~\eqref{eq:Phozaev} shows that $\lambda$ is uniquely determined by $\mu$.\\
	\par Let us now prove (ii). Denote by $\cS(\R^d)$ the Schwartz space. We define the functional $\cH:\cS(\R^d)\to\R^2$ by
	\begin{equation}
	\cH(h) = \begin{pmatrix}
	\int h(x)\phi(x)\, \dd x\\
	\int h(x) \gk\phi(x) e^{-\gk \phi^2(x)}\, \dd x.
	\end{pmatrix}
	\end{equation}
	We split the proof in two parts.\\
	\textbf{Step (a)} 
	Let us first prove that $\cH$ is surjective. To see why, assume that $\cH$ is not surjective. As $\cH$ is linear, this implies that the dimension of the range of $\cH$ is one. In particular, there is a vector $(\alpha,\beta)\in\R^2\setminus \{0\}$ such that $\cH(h)$ is orthogonal to $(\alpha, \beta)$ for all $h\in \cS(\R^d)$.
	Thus, for all $h\in\cS(\R^d)$
	\begin{equation}
	\langle h, \alpha \phi+ \beta\gk\phi e^{-\gk\phi^2}\rangle =0,
	\end{equation}
	which implies that
	\begin{equation}
	\alpha \phi(x)+ \beta\gk\phi(x) e^{-\gk\phi(x)^2}=0 \quad\text{for almost all }x\in \R^d.
	\end{equation}
	We conclude that
	\begin{equation}
	\phi(\R^d)\subseteq \Big\{0,\sqrt{-\frac{1}{\gk}{\log\big(-\frac{\alpha}{\gk\beta}\big)}}\Big\}.
	\end{equation}
Moreover, we know that $\phi$ as a minimizer is radially symmetric and strictly decreasing. This however is only possible if $\phi(x)>0$ for all $x\in\R^d$, thus $\phi(\R^d)= \sqrt{-\frac{1}{\gk}{\log(-\frac{\alpha}{\gk\beta})}}$ which contradicts the fact that $\|\phi\|_2=1$. Hence, $\cH$ is surjective.\\
	\noindent
	\textbf{Step (b)} Since $\cH$ is surjective we may pick $h\in \cS(\R^d)$ such that $\cH(h) = (0,1)$. We now define three functionals:
	\begin{equation}\label{eq:J}
	J_1(\phi) = \tfrac12 \|\nabla \phi\|_2^2,\quad J_2(\phi) = \| \phi\|_2^2,\quad J_3(\phi)= \int (1-e^{-\gk \phi(x)^2}) \dd x.
	\end{equation}
	With this choice of $h$, as $\gep\to 0$, a direct calculation shows that 
	\begin{equation}\label{eq:J3}
	J_3\Big(\frac{\phi + \gep h}{\|\phi + \gep h\|_2}\Big) = b+ 2\gep + o(\gep).
	\end{equation}
	Note furthermore that since $\phi$ and $h$ are orthogonal in $L^2$ we have that 
	\begin{equation}
	\|\phi+\gep h\|_2 = \sqrt{1+\gep^2\|h\|_2^2} >1.
	\end{equation}
	Hence, 
	\begin{equation}\label{eq:J1monotone}
	I(b+2\gep+o(\gep))\leq J_1\Big(\frac{\phi + \gep h}{\|\phi + \gep h\|_2}\Big) \leq J_1(\phi+\gep h).
	\end{equation}
	Thus, expanding $I$ and $J_1$ around $b$ and $\phi$ respectively, we see that
	\begin{equation}
	I(b+2\gep + o(\gep))= I(b) +2\gep I'(b) + o(\gep)\leq J_1(\phi)- \gep \int \Delta \phi(x)h(x)\, dx + o(\gep). 
	\end{equation}
	By substracting $I(b) = J_1(\phi)$ and dividing by $\gep$ in the previous inequality, and since $\gep$ can be positive or negative, we conclude that
	$I'(b) = - \tfrac12\int \Delta \phi(x) h(x)\, dx$. We may conclude the proof by using that $\phi$ solves the Euler-Lagrange equation and that $\cH(h)=(0,1)$.
\end{proof}
\textbf{Step 2.} \emph{Uniqueness of solution to Equation~\eqref{eq:EL1d-}}: This step relies on a uniqueness result of Serrin and Tang which we reproduce here, for the reader's convenience.
\begin{proposition}[Serrin-Tang, Theorem 1 in~\cite{SerrinTang2000}, case $m=2$]
\label{pr:serrin-tang}
The quasilinear elliptic equation
\beq
\label{eq:serrin-tang}
\Delta u + f(u) = 0 \qquad \text{in} \quad \bbR^d, \quad d\ge 3,
\eeq
admits at most one radial ground state\footnote{By \emph{ground state} we mean a non-negative non-trivial $C^1$ solution of~\eqref{eq:serrin-tang} which tends to zero at $\infty$.} if, for some $u_0>0$,
\begin{itemize}
\item[(H1)] $f$ is continous on $(0,\infty)$, with $f(u)\le 0$ on $(0,u_0]$ and $f(u)>0$ for $u>u_0$;
\item[(H2)] $f\in C^1(u_0,\infty)$, with $g(u)= uf'(u)/f(u)$ non-increasing on $(u_0,\infty)$.
\end{itemize}

\end{proposition}
Let us write (for simplicity we omit the dependence of the Lagrange multipliers on $b$):
\beq
f(r) = r\psi(r), \qquad \psi(r) := \gl + \mu \gk e^{-\gk r^2}.
\eeq 
Let us first check that Hypothesis (H1) in Proposition~\ref{pr:serrin-tang} is satisfied in our case. 
By Proposition~\ref{pr:uniq-Lagr-mult}\rm{(ii)}, we have that $\mu \le0$, so that by Proposition~\ref{pr:uniq-Lagr-mult}\rm{(i)}, $\psi(+\infty) = \lambda >0$. Next we will show that $\psi(0) = \lambda +\gk\mu<0$. First, assume $d\ge 5$ and define $h(u) = (1-u)^{\frac{2}{d}-1}\chi(u)$, where $\chi(u) := 2\gk^{2/d}I(\gk u)$ for $u\in(0,1)$, as in \cite[Eq.~(1.10)]{BBH2001}. By~\cite[Theorem 5(iii)]{BBH2001}, $h'(u)\le0$, from which we obtain
\beq
\label{eq:chi}
(1-\tfrac 2d) \chi(u)  + (1-u)\chi'(u) \le0, \qquad {\rm for\ a.e.\ } u\in(0,u^*_d).
\eeq
Combining with (i) and (ii) in Proposition~\ref{pr:uniq-Lagr-mult}, we get
\beq
\tfrac12(\gl + \gk \mu) = (1-\tfrac 2d) I(b)  + (\gk-b)I'(b) \le0, \qquad {\rm for\ a.e.\ } b\in(0,\gk u^*_d).
\eeq
Let us now discard the possibility of equality in the line above. By~\cite[Theorem 3(iii)]{BBH2001}, we also have $\tfrac 2 {du} \chi(u) + \chi'(u)\le 0$ for a.e.\ $u\in(0,1)$, so that equality in~\eqref{eq:chi} and the fact that $\chi(u)>0$ (see~\cite[Theorem 3]{BBH2001})  yields $u\ge 2/d$. This would contradict our assumption that $b=\gk u < 2\gk/d$. If $d=4$, the same argument goes through by applying~\cite[Theorem 4(ii)]{BBH2001} instead and noticing that the exponents $1-\tfrac 2 d$ and $\tfrac 2d$ coincide.
If $d=3$, we apply~\cite[Theorem 4(ii)]{BBH2001} to the function $h(u) = (1-u)^{-2/d}\chi(u)$ instead and get
\beq
0\ge \tfrac 2d I(b)  + (\gk-b)I'(b) > (1-\tfrac 2d) I(b)  + (\gk-b)I'(b), \qquad b\in(0,\gk).
\eeq
This settles our claim that $\psi(0)<0$.
Then, the equation $\psi(r)=0$ has a unique positive solution which we denote by $\ga$. One can readily check that $f$ is continuous on 	$(0, \infty)$ with $f(r)\le 0$ for $r\in(0,\ga]$ and $f(r)>0$ for $r>\ga$.\\
\par Let us now check Hypothesis (H2) in Proposition~\ref{pr:serrin-tang}, according to which the function
\beq
g(r) := \frac{rf'(r)}{f(r)} 
\eeq
should be non-increasing on $(\ga,\infty)$. By a direct computation,
\beq
g(r) = 1 + r \frac{\psi'(r)}{\psi(r)} = 1 + 2\gk r^2
\Big[ \frac{1}{1 + \frac{\gk \mu}{\gl} e^{-\gk r^2}} - 1\Big].
\eeq
By definition, $\psi(\ga)=0$, from which we get
\beq
\frac{\gk \mu}{\gl} = - e^{\gk \ga^2},
\eeq
so that
\beq
\label{eq:aux-uniq}
g(r) -1 = 2\gk r^2
\Big[ \frac{1}{1 - e^{\gk(\ga^2- r^2)}} - 1\Big].
\eeq
Let us re-parametrize the problem by setting $u=u(r):=e^{-\gk r^2} \in (0, e^{-\gk \ga^2})$, as $r\in (\ga, \infty)$. Since $u(r)$ is decreasing in $r$, we must now check that the right-hand side in \eqref{eq:aux-uniq} is non-decreasing in $u$. We get
\beq
g(r) -1 = 2(-\log u) \Big[\frac{1}{1- e^{\gk \ga^2}u} - 1 \Big],
\eeq
and
\beq
\frac{\dd}{\dd u}g(r) = \frac{2e^{\gk \ga^2}\Xi(u)}{(1- e^{\gk \ga^2}u)^2},
\qquad \Xi(u) := -\log u  - (1 - e^{\gk\ga^2}u).
\eeq
It is now straightforward to show that $\Xi(u)\ge 0$ for $0 < u < e^{-\gk\ga^2}$, which completes the proof.\\
\par We may now	conclude the proof of Theorem~\ref{thm:1}, noting that:
\begin{lemma}\label{lem:differentiability}
	The function $b\mapsto I(b)$ is almost-everywhere differentiable on $[0,\infty)$.
\end{lemma}
\begin{proof}[Proof of Lemma~\ref{lem:differentiability}]
	Note that $b\mapsto I(b)$ is a monotone function. Thus, by Lebesgue's theorem on the differentiability of monotone functions we can conclude that $I$ is almost everywhere differentiable.
	\end{proof}
\section{Proof of Theorem~\ref{thm:2}: Tube property}
\label{sec:tube}
In this section we prove Theorem~\ref{thm:2} assuming the validity of Proposition~\ref{pr:LDP-measure} and the first statement in Proposition~\ref{prop:restGamma-lsc}.

\begin{lemma}[Scaling properties]
\label{lem:scaling}
Let $\phi\in H^1(\bbR^d)$. For $a\in (0,\infty)$, define $\phi_{[a]}(x):= a^{d/2}\phi(ax)$. Then,
\beq
\|\phi_{[a]}\|_2 = \|\phi\|_2, \qquad \|\nabla\phi_{[a]}\|_2 = a \|\nabla\phi\|_2.
\eeq
Moreover, the function
\beq
a\in(0,\infty) \mapsto \int_{\bbR^d} \Big(1-e^{-\gk \phi_{[a]}^2(x)}\Big) \dd x
\eeq
is continuous, non-increasing and converges to 0 as $a\to \infty$, and to $\gk\|\phi\|_2^2$ as $a\to 0$.
\end{lemma}
\begin{proof}[Proof of Lemma~\ref{lem:scaling}]
We refer the reader to~\cite[Section 5, proof of Lemma 12]{BBH2001}.
\end{proof}
The limit $a\to 0$ in Lemma~\ref{lem:scaling} corresponds to {\it evanescent mass} and better enlightens our definition of $\tilde \Gamma$ on $\tilde\cX$ in~\eqref{eq:def-tilde-JG}.

\begin{lemma}
	\label{lem:positive}
Let $b\in(0,\kappa)$ if $d=3$ or $b\in(0,\tfrac 2 d \kappa)$ if $d\ge 4$ be such that the variational problem in~\eqref{def:Ifunction} has a unique minimizer modulo space shift. Let $\cU(\fm_b)$ be an open neighbourhood of $\fm_b$ w.r.t.\ the $\bfD$-topology. Then,
\beq\label{eq:positive}
\inftwo{\xi \in \tilde \cD(b)}{\xi \notin \cU(\fm_b)} \tilde J(\xi)
>\inf_{\xi \in \tilde\cD(b)} \tilde J(\xi)\,.
\eeq
\end{lemma}
\begin{proof}[Proof of Lemma~\ref{lem:positive}] Assume that the left-hand side in~\eqref{eq:positive} is finite, otherwise there is nothing to prove. Then, we may safely restrict the infimum on the left-hand side to $\tilde \cD_K(b) := \tilde \cD(b) \cap \{\xi \in \tilde \cX\colon \tilde J(\xi)\le K\}$ for some $K>0$. By Proposition~\ref{prop:restGamma-lsc}, $\tilde \cD_K(b)$ is a closed set. Moreover, $(\tilde \cX, \bfD)$ is a compact metric space, therefore there exists a minimizer, further denoted by $\xi$, for the function $\tilde J$ on the compact set $\tilde\cD_K(b) \cap \cU(\fm_b)^c$. We distinguish between cases according to the number of elements in $\xi$.\\
{\bf \noindent Case 0.} Assume that $\xi = \emptyset$. Then, $\tilde\Gamma(\xi)=\gk>b$, hence $\xi \notin \tilde\cD(b)$, which contradicts our assumption.\\
{\bf \noindent Case 1.} Assume that $\xi$ has a single element, i.e.\ $\xi = \{\tilde\ga\}$ for some sub-probability measure $\ga$. Necessarily, $\ga(\dd x) = \phi^2(x)\dd x$ for some $\phi\in H^1(\bbR^d)$, otherwise $\tilde J(\xi) = +\infty$. There are then two further subcases:\\
{\bf \noindent Case 1a.}
If $\int \phi^2 = 1$, then by Theorem~\ref{thm:1} (uniqueness of the minimizer among the set of probability measures modulo space shifts), $\tilde J(\xi) > \tilde J(\fm_b)$, which closes this case.\\
{\bf \noindent Case 1b.}
Assume now that $\int \phi^2 \in (0,1)$. 
Since $\xi\in \tilde \cD(b)$, we know that
\beq
\label{eq:cstr}
\int (1- e^{-\gk \phi^2}) + \kappa\Big(1-{\int} \phi^2\Big) \le b\,.
\eeq
Assume that the left hand side in~\eqref{eq:cstr} is strictly smaller than $b$. Arguing as in~\cite[Proof of Lemma 12]{BBH2001}, we consider the map
\beq
\ba
a\in(0,\infty) \mapsto \int \Big(1-e^{-\gk \phi_{[a]}^2(x)}\Big) \dd x + \kappa \Big(1-\int \phi_{[a]}^2(x) \dd x\Big),
\ea
\eeq 
which is continuous by Lemma~\ref{lem:scaling}. Therefore, there exists $\eta>0$ such that for all $a\in(1-\eta, 1)$, we have
\beq\label{eq:elementofDb}
\int \Big(1-e^{-\gk \phi_{[a]}^2(x)}\Big) \dd x + \kappa \Big(1-\int \phi_{[a]}^2(x) \dd x\Big)\leq b\,,
\eeq
hence $\phi_{[a]}\in \tilde\cD(b)$.
Note that, by Lemma~\ref{lem:scaling},
\beq
\|\phi_{[a]}\|_2 = \|\phi\|_2, \qquad \|\nabla\phi_{[a]}\|_2 = a \|\nabla\phi\|_2 < \|\nabla \phi\|_2\, \qquad a\in (1-\eta,1),
\eeq
in which case~\eqref{eq:positive} is satisfied, since
\beq
\inf_{\xi \in \tilde\cD(b)} \tilde J(\xi) \le \frac12\|\nabla\phi_{[a]}\|_2^2 < \frac12 \|\nabla \phi\|_2^2.
\eeq
Assume now that the constraint in~\eqref{eq:cstr} is saturated,
which we may rewrite as
\beq
\label{eq:cstr-bis}
\int (e^{-\kappa \phi^2} - 1 + \kappa \phi^2) = \kappa - b.
\eeq
If there exists another element $\xi'\in\tilde\cD(b)$ such that $\tilde J(\xi')< \tilde J(\xi)$ then obviously the infimum of $\tilde J$ over $\tilde \cD(b)$ is strictly smaller than $\tilde J(\xi)$ and there is nothing more to prove. Therefore, we may assume from now on that $\phi$ is a local minimizer of $\|\nabla\phi\|_2$ under the constraint in~\eqref{eq:cstr-bis}, and write the associated Euler-Lagrange equation.
\begin{itemize}
\item In the case $d\in\{3,4\}$, it follows from a slight adaptation of~\cite[Lemma 14]{BBH2001} that $\int \phi^2 = +\infty$, which contradicts our assumption.
\item In the case $d\ge 5$, let us define $u := b/\kappa \in (0,1)$ and $\psi^2(x):= \kappa \phi^2([\kappa (1-u)]^{1/d} x)$. Then, by a straightforward change of variable, $\int \psi^2 = (1-u)^{-1}\int \phi^2 < (1-u)^{-1}$ and $\psi$ is a minimizer of $(1-u)^{1-2/d}\|\nabla\psi\|_2^2$ under the constraint $\int e^{-\psi^2}-1+\psi^2 = 1$. However, by the definition of $u_d^*$ we have that $\|\psi\|_2^2 \ge (1-u_d^*)^{-1}$, which is a contradiction to the fact just established, namely $\int \psi^2  < (1-u)^{-1}$, whenever $u\le u_d^*$, i.e.\ $b\le \kappa u_d^*$.
\end{itemize}
{\bf \noindent Case 2.}  Finally, let us assume that $\xi$ contains at least two elements, i.e. (i) at least two distinct elements, or (ii) at least one element with multiplicity at least two. Let us denote them by $\tilde \ga_1$ and $\tilde \ga_2$. In the sequel we pick $\ga_i(\dd x) = \phi_i^2(x)\dd x$, with $\phi_i \in H^1(\bbR^d)$ for $i\in\{1,2\}$, two elements of $\tilde \ga_1$ and $\tilde \ga_2$ such that $\min(\phi_1^2, \phi_2^2)>0$ on a set of positive measure. Consider $\ga_3 := \ga_1 + \ga_2$, and
\beq
\xi' := \{\tilde\ga_3,\, \xi\setminus\{\tilde\ga_1, \tilde \ga_2\}\}.
\eeq
By the convexity inequality for $J$, see e.g.~\cite[Theorem~7.8]{LiebLoss}, $J(\ga_3) \le J(\ga_1) + J(\ga_2)$, hence $\tilde J(\xi') \le \tilde J(\xi)$. To complete the argument, let us first notice that for every $u,v\in[0,1]$, $1- uv \le (1-u) + (1-v)$, with the inequality being strict as soon as $\max(u,v)<1$. By our choice of $\ga_1$ and $\ga_2$, this yields
\beq
\int 1 - e^{-\kappa(\phi_1^2 + \phi_2^2)} <
\int (1 - e^{-\kappa \phi_1^2}) +
\int (1 - e^{-\kappa \phi_2^2}),
\eeq
hence $\tilde\Gamma(\xi')< \tilde\Gamma(\xi)\le b$. Let us now define $\ga_4(\dd x) := a^{d}\ga_3(a\, \dd x)$ and $\xi'' := \{\ga_4, \xi'\setminus\{\ga_3\}\}$. By using Lemma~\ref{lem:scaling} and choosing $a\in (0,1)$ close enough to one, we obtain $\tilde\Gamma(\xi'')\le b$ and $\tilde J(\xi'')<\tilde J(\xi') \le \tilde J(\xi)$, which completes the proof.
\end{proof}
 The reader may check that the above proof actually yields the following:
\begin{corollary}
	\label{cor:unique}
Under the same assumptions as in Theorem~\ref{thm:2}, $\fm_b$ is the unique minimizer of the rate function $\tilde J$ in $\tilde\cD(b)$.
\end{corollary}

We may now prove Theorem~\ref{thm:2}. Let $\cU(\fm_b)$ be an open neighbourhood of $\fm_b$ with respect to the $\bfD$-topology. By Lemma~\ref{lem:positive}, the quantity
\beq
\gd := \inftwo{\xi \in \tilde \cD(b)}{\xi \notin \cU(\fm_b)} \tilde J(\xi)
-\inf_{\xi \in \tilde\cD(b)} \tilde J(\xi)
\eeq
is (strictly) positive. By Proposition~\ref{pr:LDP-measure}, there exists $\gep_0$ such that, for all $\gep \in (0,\gep_0)$,
\beq
\limsup_{n\to\infty} n^{\frac2d -1}
\log \bP(\tilde L_{M,\gep} \notin \cU(\fm_b), R_n\le bn) \le 
- { \frac 1d}\inftwo{\xi \in \tilde \cD(b)}{\xi \notin \cU(\fm_b)} \tilde J(\xi) + \frac{\gd}{2{ d}}.
\eeq
 Using Theorem~\ref{thm:MDP-rw} and the fact that the infimum of $J$ over $\cD(b)$ coincides with the infimum of $\tilde J$ over $\tilde \cD(b)$ (Corollary~\ref{cor:unique}) we obtain
\beq
\limsup_{n\to\infty} n^{\frac2d -1}
\log \bP(\tilde L_{M,\gep}\notin \cU(\fm_b) | R_n\le bn) \le - { \frac 1d}\inftwo{\xi \in \tilde \cD(b)}{\xi \notin \cU(\fm_b)} \tilde J(\xi) + { \frac 1d}\inf_{\xi \in\tilde\cD(b)}\tilde J(\xi) + \frac{\delta}{2{ d}},
\eeq
which is less that ${ -\gd/(2d)}<0$.
This completes the proof of Theorem~\ref{thm:2}.

\section{Proof of Proposition~\ref{pr:LDP-measure}: Large deviation upper bound}
\label{sec:ldp}
This section is devoted to the proof of the large deviation upper bound in the $\bfD$-topology, which is key to obtain Theorem~\ref{thm:2}. In Section~\ref{sec:cond-range} we first reduce the deviations for the volume of the range of the random walk to the deviations of its expectation conditioned to the skeleton and show that the conditional expectation may be expressed as a certain functional of the skeleton pair empirical measure. That first step closely follows~\cite{Phetpradap, BBH2001}, the only difference being that we do not fold the random walk on a torus. Our compactification method rather relies on the use of the compact metric space $(\tilde \cX, \bfD)$ instead of the more standard weak topology on the space of probability measures. To be more precise, an adaptation of the $\bfD$-topology is needed to obtain a large deviation principle for the {\it pair} empirical measure. This will be provided in Section~\ref{sec:pairLDP}, using previous work of the authors~\cite{EP23}. The continuity properties that are needed to apply the standard contraction principle to the relevant functionals are given in Section~\ref{sec:lsc}. The main differences with~\cite{Phetpradap, BBH2001} will be discussed at the beginning of that section. All these ingredients will be combined in Section~\ref{sec:conclusion-proof} to finally prove Proposition~\ref{pr:LDP-measure}. Remaining sections contain the deferred proofs of more technical and lengthy lemmas.
\subsection{Approximation of the conditional range}
\label{sec:cond-range}
The very first step in our way to~Proposition~\ref{pr:LDP-measure} is the use of the following concentration inequality, adapted from~\cite[Proposition 2.2.2]{Phetpradap}.
\begin{proposition}[Concentration inequality]
\label{pr:conc.ineq}
For all $\gd>0$,
\beq
\lim_{\gep\to 0} \limsup_{n\to\infty} n^{\frac2d -1} \log
\bP(|R_n - \bE(R_n | \hat S)| \ge \gd n) = -\infty.
\eeq
\end{proposition}
As a consequence, we may restrict our attention to the conditional expectation of the volume of the range. The only difference with~\cite[Proposition 2.2.2]{Phetpradap} is that the random walk has not been folded on the torus. Hence only mild modifications are needed to get Proposition~\ref{pr:conc.ineq}. Those are deferred to Appendix~\ref{sec:app-con-ine}.

The next step is to express the conditional expectation as a function of the skeleton empirical measure. Although this step follows exactly~\cite{Phetpradap, BBH2001}, we shortly reproduce the computation here, for the reader's convenience and to make the paper more self-contained. More notation is needed beforehand: we introduce the auxiliary functions
\beq
q_{\ell,\gL}(x,y) := \bP_x(S_{(0,\ell]} \cap \gL \neq \emptyset	 | S_\ell = y),
\qquad x,y\in \bbZ^d,
\qquad \gL \subset \bbZ^d,
\eeq
\beq
\bar q_{\ell,\gL}(x,y) := q_{\ell,\gL}(\lfloor xn^{1/d}\rfloor,\lfloor yn^{1/d}\rfloor),
\qquad x,y\in \bbR^d,
\qquad \gL \subset \bbZ^d,
\eeq
and when $\gL=\gL(n,z)$ is the singleton $\{\lfloor zn^{1/d}\rfloor\}$ ($z\in\bbR^d$) we define
\beq
\bar q_{\ell}(x,y;z) := \bar q_{\ell,\gL(n,z)}(x,y),
\qquad x,y\in \bbR^d.
\eeq
Let us now define
\beq
\frak{q}_{\gep, n, \gL}(x,y):= -M \log[1-\bar q_{\ell, \gL}(x, y)], \qquad x,y\in \bbR^d\,,
\eeq
where we remind the reader of the definition of $M$ in~\eqref{eq:M}.
Let $z\in\bbR^d$. When $\gL=\gL(n,z)$ is the singleton $\{\lfloor zn^{1/d}\rfloor\}$, we define
\beq
\label{def:frakq}
\frak{q}_{\gep,n}(x,y;z) := \frak{q}_{\gep,n,\gL(n,z)}(x,y)
\qquad x,y\in \bbR^d. 
\eeq
With all this notation in hand, we may finally write:
\beq
\label{eq:exp-cond-range}
\ba
\frac1n \bE(R_n | \hat S) &= \frac1n \sum_{z\in\bbZ^d} [1-\bP(z\notin S_{(0,n]} | \hat S)]\\
&= \frac1n \sum_{z\in\bbZ^d} \Big[ 1- \prod_{1\le i \le M} \bP(z\notin S_{((i-1)\ell, i\ell]} | \hat S)\Big]\\
&=\frac1n \sum_{z\in\bbZ^d} \Big[1-  \prod_{1\le i \le M} [1-q_{\ell,\{z\}}(\hat S_{i-1}^{(\gep)} n^{1/d}, \hat S_i^{(\gep)} n^{1/d})] \Big]\\
&= \int_{\bbR^d} \dd z\Big[ 1- \prod_{1\le i \le M}[ 1- \bar q_{\ell}(\hat S_{i-1}^{(\gep)}, \hat S_i^{(\gep)} ; z)]\Big]\\
&=\int_{\bbR^d} \dd z
\Big[
1- \exp\Big(- \int \frak{q}_{\gep,n}(x,y;z) L_{M,\gep}^{(2)}(\dd x, \dd y)\Big)
\Big].
\ea
\eeq

In what follows, we replace the function $\frak{q}_{\gep,n}$ by a function that does not depend on $n$ anymore, see Proposition~\ref{pr:app-cond-range} below. First, let us define for $x,y\in\bbR^d$,
\beq\label{def:varphi}
\varphi_\gep(x,y) := \int_0^\gep \dd s \frac{p_{s/d}(-x)p_{(\gep-s)/d}(y)}{p_{\gep/d}(y-x)}.
\eeq
where $p_s$ denotes the Brownian heat kernel, i.e.\
\beq
\label{eq:Bheatk}
p_s(x) = \frac{1}{(2\pi s)^{d/2}}\exp\Big(-\frac{|x|^2}{2s}\Big),
\qquad x\in\bbR^d,
\eeq
and for all $\mu\in \cM_1(\bbR^d\times\bbR^d)$, 
\beq
\label{eq:def-phi-infty}
\phi_{\infty,\gep}(\mu)
:=\int  \dd z
\Big(1 - \exp\Big\{-\frac{\gk_d}{\gep} \int \mu(\dd x, \dd y) \varphi_\gep(x-z,y-z)\Big\}\Big).
\eeq

Furthermore we shall restrict the pair empirical measure to a certain subset of $\cM_1(\bbR^d\times \bbR^d)$ on which we control the maximal distance run by the random walk along each ``edge" of the skeleton. Namely, we define 
\beq
\label{eq:def-cM-A-gep_0-first}
\cM(A,\gep_0) :=
\{\mu\in\cM_1(\bbR^d\times \bbR^d) \colon
\mu(\{(x,y)\colon |x-y| > A\})  \le \gep_0\}, \qquad
A,\gep_0 >0.
\eeq
This restriction is harmless due to the following:
\begin{proposition}
\label{pr:restricy-PEM}
Let $\gep, \gep_0 >0$ and $b\in(0,\gk)$ be fixed. Then,
\beq
\lim_{A\to +\infty} \limsup_{n\to \infty} n^{\frac2d - 1}\log\bP(L_{M,\gep}^{(2)} \notin \cM(A,\gep_0)  | R_n \le bn) = -\infty.
\eeq
\end{proposition}

\begin{proof}[Proof of Proposition~\ref{pr:restricy-PEM}]
We use the rough upper bound
\beq
\label{eq:restricy-PEM}
\bP(L_{M,\gep}^{(2)} \notin \cM(A,\gep_0)  | R_n \le bn) \le
\frac{\bP(L_{M,\gep}^{(2)} \notin \cM(A,\gep_0))}{\bP(R_n \le bn)}.
\eeq
Note that
\beq
\bP(L_{M,\gep}^{(2)} \notin \cM(A,\gep_0))\leq \bP(\card\{0<i\le M \colon |S_{i\ell} - S_{(i-1)\ell}| > A\sqrt{\tfrac{\ell}{\gep}}\} > \gep_0 M).
\eeq
Notice that for each $C>0$ (to be determined later) there exists $A$ such that
\beq
\sup_{\ell \ge 1}\bP\Big(|S_\ell| > A\sqrt{\tfrac{\ell}{\gep}}\Big) \le \frac{\gep_0}{C}.
\eeq
For such value of $A$ we get the binomial estimate:
\beq
\bP(L_{M,\gep}^{(2)} \notin \cM(A,\gep_0))
\le \bP(\bin(M, \gep_0/C) > \gep_0M).
\eeq
By a standard Large Deviation estimate (see Lemma~\ref{lem:LDest}), we get (recall that $\gep_0$ is fixed)
\beq
\bP(L_{M,\gep}^{(2)} \notin \cM(A,\gep_0))
\le \exp(-[1+o(1)] \gep_0\Xi(C)\tfrac 1\gep n^{1-2/d}),
\eeq
where $\lim_{C\to \infty}\Xi(C) = +\infty$ and the $o(1)$ holds as $n\to \infty$. In view of~\eqref{eq:restricy-PEM} and Theorem~\ref{thm:MDP-rw}, it is enough to choose $C$ large enough such that $\gep_0\Xi(C) > (\gep/d) I(b)$, and then pick $A$ accordingly, in order to conclude.
\end{proof}
We may finally state the main result of this subsection. The proof is deferred to~Appendix~\ref{sec:app-proof-cond}, due to its length.
\begin{proposition}\label{pr:app-cond-range} Let $\gep>0$. For every $\gep,\gep_0, A >0$, there exists $n_0(\gep,\gep_0, A)$ such that for all $n\ge n_0(\gep,\gep_0, A)$,
\beq
\tfrac1n \bE(R_n|\hat S) \ge \phi_{\infty,\gep}(L_{M,\gep}^{(2)})-6\gep_0,
\eeq
on the event $L_{M,\gep}^{(2)}\in \cM(A,\gep_0)$.
\end{proposition}

\begin{remark}
Proposition~\ref{pr:app-cond-range} partially translates~\cite[Lemma 2.2.10 (a)]{Phetpradap} to the infinite volume case. However, our proof will deviate from the one in~\cite{Phetpradap} exactly to avoid issues coming from the fact that the random walk considered here is on $\R^d$ and not on a torus.
\end{remark}
Let us end this section with an observation that shall become useful later on
\begin{lemma}
\label{lem:phi.q.prob}
For all $x,y\in\bbR^d$ and $n\in\bbN$, $z\in\bbR^d\mapsto \frac1\gep \varphi_\gep(x-z,y-z)$ is the density of a probability measure on $\bbR^d$ and $z\in\bbR^d \mapsto \frac 1\gep n^{1-2/d} \bar q_\ell(x,y;z)$ is a sub-probability measure.
\end{lemma}
\begin{proof}[Proof of Lemma~\ref{lem:phi.q.prob}]
The first part of the statement is straightforward since for all $s\in (0,\gep)$,
\beq
\label{eq:phi.q.prob}
\int \dd z\ p_{s/d}(z-x) p_{(\gep-s)/d}(y-z) = p_{\gep/d}(y-x),
\eeq
by the Chapman-Kolmogorov equation. For the second part of the statement, we write
\beq
\ba
n^{1-2/d}\int \dd z\ \bar q_\ell(x,y;z) 
&= n^{1-2/d}\int \dd z\ 
\bP_{\lfloor x n^{1/d}\rfloor}(S_{(0,\ell]}\cap \{\lfloor z n^{1/d}\rfloor\}\neq \emptyset | S_\ell = \lfloor yn^{1/d}\rfloor)\\
&= n^{-2/d} \sum_{z\in\bbZ^d}
\bP_{\lfloor x n^{1/d}\rfloor}(S_{(0,\ell]}\cap \{z\}\neq \emptyset | S_\ell = \lfloor yn^{1/d}\rfloor)\\
&\le n^{-2/d} \sum_{z\in\bbZ^d} \sum_{1\le i \le \ell} \bP_{\lfloor x n^{1/d}\rfloor}(S_i = z | S_\ell = \lfloor yn^{1/d}\rfloor) = n^{-2/d} \ell = \gep.
\ea
\eeq
\end{proof}
%
%

\subsection{Large deviation principle for pair empirical measures modulo shifts} 
\label{sec:pairLDP}
With Proposition~\ref{pr:app-cond-range} in hand it is only natural to apply a large deviation principle for the skeleton {\it pair} empirical measure. The Mukherjee-Varadhan topology introduced in Section~\ref{sec:compact-topo0} was adapted in~\cite{EP23} to allow strong large deviation principles for {\it pair} empirical measures. In this section we explain the modifications needed to obtain this topology from the one in Section~\ref{sec:compact-topo0}.
\noindent
Let $\cM_1^{(2)} := \cM_1(\bbR^d\times \bbR^d)$ be the space of probability measures on $\bbR^d\times \bbR^d$ and $\cM_{\le1}^{(2)} := \cM_{\le1}(\bbR^d\times \bbR^d)$ be the space of sub-probability measures on $\bbR^d\times \bbR^d$. We consider the action of the shifts $\gt_{x,x}$ for $x\in\bbR^d$, defined by:
\beq
\int_{\bbR^d\times \bbR^d} f(u,v) (\gt_{x,x} \nu)(\dd u,\dd v) = \int_{\bbR^d\times \bbR^d} f(u+x,v+x)  \nu(\dd u,\dd v)
\eeq
for all continuous bounded functions $f\colon \bbR^d\times \bbR^d \mapsto \bbR$ and $\nu\in \cM_{\le1}^{(2)}$. We shall denote by $\tilde \cM_1^{(2)}$ (resp. $\tilde \cM_{\le 1}^{(2)}$) the space of equivalence classes of $\cM_1^{(2)}$ (resp. $\cM_{\le 1}^{(2)}$) under the collection of shifts $\gt_{x,x}$. For any $\alpha\in \cM_1^{(2)}$ we denote by $\tilde\alpha$ its orbit, i.e., equivalence class. Recall the definition of $\cF_k$ in Section~\ref{sec:compact-topo0}. For $k\ge 1$, $f\in \cF_{2k}$ and $\ga \in \cM_{\le 1}^{(2)}$, we write
\beq
\label{eq:defLambda}
\gL(f, \ga) := \int f(u_1, v_1,\ldots, u_{k}, v_{k}) \prod_{1\le i\le k} \ga(\dd u_i, \dd v_i),
\eeq
which only depends on the orbit $\tilde \ga$. This time, we define
\beq
\cF^{(2)} := \bigcup_{k\ge 2} \cF_{2k},
\eeq
for which there exists a countable dense set (under the uniform metric) denoted by
\beq
\{f^{(2)}_r(u_1, v_1,\ldots, u_{k_r}, v_{k_r})\colon r\in \bbN\},
\eeq
see~\cite[Section 2.2]{MV2016}. We define
\beq
\tilde\cX^{(2)} := \Big\{
\xi = \{\tilde \ga_i\}_{i\in I} \colon \ga_i \in \cM_{\le 1}^{(2)},\ \sum_{i\in I} \ga_i(\bbR^d \times \bbR^d) \le 1\Big\},
\eeq
where $I$ may be empty, finite or countable. For any $\xi_1, \xi_2\in \tilde \cX^{(2)}$, define
\beq
\mathbf{D}_2(\xi_1, \xi_2) := \sum_{r\ge 1} \frac{1}{2^r} \frac{1}{1+\|f^{(2)}_r\|_{\infty}}
\Big|%
\sum_{\tilde\ga\in\xi_1} \gL(f^{(2)}_r, \ga)- \sum_{\tilde\ga\in\xi_2} \gL(f^{(2)}_r, \ga)
\Big|.%
\eeq
It was then shown in~\cite{EP23} that $\tilde\cX^{(2)}$ equipped with $\mathbf{D}_2$ is a compact metric space, and that $\tilde \cM_1^{(2)}$ is dense in $\tilde\cX^{(2)}$.\\

\par In what follows, $(\pi_\gep)$ denotes the Brownian semigroup, i.e.\ for $\gep>0$ and $x\in\bbR^d$,
\beq
\pi_\gep(x,\dd y) := p_\gep(x,y)\dd y, \qquad p_\gep(x,y):=p_\gep(y-x),
\eeq
(recall~\eqref{eq:Bheatk}) and for all $\ga\in\cM_{\le 1}(\bbR^d)$ we write
\beq
(\ga \otimes \pi_\gep)(\dd x, \dd y):= \ga(\dd x)\pi_\gep(x,\dd y) \in \cM_{\le 1}(\bbR^d \times \bbR^d).
\eeq
We denote by $h$ the relative entropy, defined by
\begin{equation}
h(\mu|\nu)= 
\begin{cases}
\int \log\Big(\frac{\dd \mu}{\dd \nu}\Big) \dd \mu & \textrm{if } \mu \ll \nu;\\
+\infty& \textrm{else;}
\end{cases}
\qquad \mu, \nu\in \cM_{\leq 1}(\bbR^d \times \bbR^d).
\end{equation}
If $\xi = \{\tilde \ga _i\}_{i\in I}\in \tilde\cX^{(2)}$, we denote by $\ga_{i,1}$ and $\ga_{i,2}$ the projections of $\ga_i$ (that is an arbitrary element of the orbit $\tilde \ga_i$) onto the first and last $d$ coordinates respectively. Then, we define, for every $\gep>0$:
\beq\label{eq:ratefct}
\tilde J_\gep^{(2)}(\xi) := \sum_{i\in I} h(\ga_i | \ga_{i,1} \otimes \pi_\gep)
\eeq
if $\ga_{i,1}=\ga_{i,2}$ for all $i\in I$, and $\tilde J^{(2)}(\xi) := +\infty$ otherwise.
The following result is key to our analysis:
\begin{proposition}[Pair empirical LDP upper bound, Theorem 7.2 in~\cite{EP23}]
\label{pr:LDPpair}
For any closed set $F$ in $\tilde\cX^{(2)}$,
\beq
\limsup_{n\to\infty}{n^{\frac2d -1}}\log \bP(\tilde{L}_{M,\gep}^{(2)}\in F )
\leq -\inf_{\xi\in F}\frac{1}{\gep} \tilde J^{(2)}_{ \gep/d} (\xi)\,.
\eeq
\end{proposition}
Before we come to the next section we will gather more facts and give ourselves a slightly more convenient notation for projections.
For all $\mu\in\cM_{\le 1}(\bbR^d\times \bbR^d)$, we define
\beq
\tilde\Pr(\tilde\mu)= \{\Pr(\mu)*\delta_x\colon\, x\in \R^d\}\,,
\eeq
where $\Pr$ denotes the usual projection onto the first $d$ coordinates of a sub-probability measure in $\R^d\times\R^d$. Given $\xi=\{\tilde{\alpha_i}\}_{i\in I}\in \tilde\cX^{(2)}$, we define 
\beq
\tilde\Pr(\xi)= \{\tilde\Pr(\tilde\alpha_{i})\}_{i\in I}\,.
\eeq
For later use we define the rate function that governs the large deviations of the empirical measure and we collect some facts about it. For a sub-probability measure $\alpha\in\cM_{\leq 1}(\R^d)$ we define
\beq
J_\gep(\alpha)=\inf\{J_\gep^{(2)}(\mu)\,:\, \Pr(\mu)=\alpha\},
\eeq
where $J_\gep^{(2)}(\mu) := h(\mu | \Pr(\mu)\otimes \pi_\gep)$ for every $\mu \in \cM_{\le 1}^{(2)}$,
and we set for $\xi\in\tilde\cX$
\beq
\tilde J_\gep(\xi)=\sum_{\alpha\in\xi}J_\gep(\alpha)\,.
\eeq
It was noted in~\cite[Lemma 8.6]{EP23} that one has the representation 
\beq\label{eq:contractionP}
\tilde J_\gep(\xi)=\inf\{\tilde J_\gep^{(2)}(\theta)\,:\, \tilde\Pr(\theta)=\xi\}\,.
\eeq
It was also noted in the proof of the same lemma that one can also write
\beq\label{eq:varrepresentation}
J_\gep(\alpha)=\sup_v \int \log\Big(\frac{v}{\pi_\gep v}\Big)\,\dd\alpha\,,
\eeq
where the supremum runs over all bounded and Borel-measurable functions or also over all continuous and compactly supported functions $v:\R^d\to [1,\infty)$.
The definition of $\tilde J_\gep$ comes from the contraction principle applied to $\tilde J_\gep^{(2)}$. However,
it turns out, see~\cite{EP23}, that $\tilde\Pr$ is not a continuous function from $(\tilde \cX^{(2)},\bfD_2)$ to $(\tilde \cX,\bfD)$. Fortunately it can be well approximated by certain continuous maps. To make this more precise we need to introduce more notation. For $A>0$ we denote by $\Xi_A:\R^d\to[0,1]$ a smooth map satisfying
\beq
\ba
\Xi_A(x)= \begin{cases}
	1, &|x|\leq A\,,\\
	0, &|x|\ge 2A\,.
\end{cases}
\ea
\eeq
We then define for any sub-probability measure $\mu\in \cM_{\le 1}(\R^d\times\R^d)$
\beq
({\Pr}_A \mu)(B)=\int_{v\in B}\int_{u\in\R^d}\Xi_A(u-v)\,\mu(\dd u,\dd v)\,,
\eeq
which, as it is the case for $\tilde\Pr$, can be extended to a map $\tilde{\Pr}_A:\tilde\cX^{(2)}\to\tilde\cX$. It was shown in~\cite[Proposition 8.1]{EP23} that $\tilde{\Pr}_A$ is continuous. To compare $\tilde\Pr$ and $\tilde{\Pr}_A$ we need one more observation. For $\xi\in\tilde\cX^{(2)}$ define
\beq
\label{eq:pre-def-cM-A-gep0}
\xi(\{(x,y)\in\R^d\times\R^d\,:\, |x-y|> A\})=\sum_{\alpha\in\xi}\alpha(\{(x,y)\in\R^d\times\R^d\,:\, |x-y|> A\})\,,
\eeq
and with a slight abuse of notation we denote
\beq
\label{eq:def-cM-A-gep0}
\cM(A,\gep_0)=\{\xi\in\tilde\cX^{(2)}\,:\, \xi(\{(x,y)\in\R^d\times\R^d\,:\, |x-y|> A\})\leq \gep_0\}\,.
\eeq
(This definition actually extends~\eqref{eq:def-cM-A-gep_0-first}). The following result was shown in
Lemmas 8.2 and 8.3 in~\cite{EP23}.
\begin{lemma}
	\label{lem:Projection}
Using the notation introduced above the following holds:
\begin{itemize}
	\item[\rm{(1)}] $\cM(A,\gep_0)$ is closed.
	\item[\rm{(2)}] For every $A,\gep_0>0$:
$	
\cM(A,\gep_0)\subseteq \{\xi\in\tilde\cX^{(2)}\,:\, \bfD(\tilde \Pr(\xi), \tilde{\Pr}_A(\xi))\leq 2\gep_0\}\,.	
$
\end{itemize}
\end{lemma}
However, it turns out that the proof of~\cite[Lemma 8.2]{EP23}, which corresponds to Item (1) in Lemma~\ref{lem:Projection}, contains a mistake. The mistake can be traced back to Equation (8.19) therein, whose correct formulation should be
\beq
\ga(\cD_A^c) \ga(\bbR^d\times \bbR^d) = \suptwo{A< A_1< A_2}{B>0} I(\ga; A_1, A_2, B).
\eeq
(Also, the definition of $\cM(A,\gep_0)$ in Equation (8.15) therein should anyway be replaced by Equation~\eqref{eq:def-cM-A-gep0} in the present paper, since the total mass of an element in $\tilde\cX^{(2)}$ can be any number in $[0,1]$). We use this opportunity to correct that mistake. It is actually possible to fix it by changing the definition in~\eqref{eq:pre-def-cM-A-gep0} and adapting the proof of~\cite[Lemma 8.3]{EP23} accordingly, but we provide a corrected proof that leaves~\eqref{eq:def-cM-A-gep0} as it is.
\begin{proof}[Proof of Item (1) in Lemma~\ref{lem:Projection}]
Let $(\xi_m)_{m\ge 1}$ be a sequence in $\cM(A,\gep_0)$ converging to $\xi \in \tilde \cX^{(2)}$ w.r.t.\ the metric $\mathbf{D}_2$. Let us show that $\xi \in\cM(A,\gep_0)$. 

{\it \noindent Step 1.} We first treat 
the simpler case when for every $m\ge 1$, $\xi_m$ consists of a single sub-probability orbit denoted by $\tilde \mu_m$. Let $\gd\in (0,\gep_0)$. Let us assume that $\xi$ is not empty (otherwise our claim is trivial) and let $M\ge 1$, $\{\tilde \ga_1, \ldots, \tilde \ga_M\}\subseteq \xi$ be such that
\beq
\label{eq:item1LemmaProj-aux}
\sum_{\tilde \ga \in \xi\setminus \{\tilde \ga_1, \ldots, \tilde \ga_M\}} \ga(\bbR^d\times \bbR^d) \le \gd.
\eeq
By~\cite[Proposition 2.12]{EP23} (more precisely, the decomposition used in the second step of the proof), we may write for some sub-probability measures $(\ga_{m,j})_{1\le j\le M, m\ge 1}$, $(\gb_m)_{m\ge 1}$ and $d$-dimensional vectors $(a_{m,j})_{1\le j\le M, m\ge 1}$, at least along a subsequence,
\beq
\mu_m= \sum_{1\le j \le M} \ga_{m,j} * \gd_{(a_{m,j}, a_{m,j})} + \gb_m, 
\eeq
with
\begin{itemize}
\item $\lim_{m\to \infty} |a_{m,i} - a_{m,j}|=\infty$ for every $i\neq j$;
\item $\ga_{m,j} * \gd_{(a_{m,j}, a_{m,j})} \stackrel{\textrm{(weakly)}}{\longrightarrow} \ga_j $, as $m\to \infty$, for every $1\le j\le M$;
\item the $\ga_{m,j} * \gd_{(a_{m,j}, a_{m,j})}$'s are \emph{widely separated} from $\gb_m$, as $m\to\infty$ (although we do not use this notion here, we refer the reader to \cite[Section 2.2]{EP23} for a definition).
\end{itemize} 
Note that the set $\cD_A := \{(u,v)\in \bbR^d\times \bbR^d\colon |u-v|\le A\}$ is closed and its complement is thus open. Also, those sets are invariant w.r.t.\ to \emph{diagonal} shifts $(u,v)\in \bbR^d \mapsto (u+x, v+x)$, for every $x\in\bbR^d$. By using the Portmanteau theorem and~\eqref{eq:item1LemmaProj-aux}, we obtain
\beq
\ba
\gep_0 \ge \liminf_{m\to\infty} \mu_m(\cD_A^c) = &\liminf_{m\to\infty}
\sum_{1\le j \le M} \ga_{m,j}(\cD_A^c) + \gb_m (\cD_A^c)\\
&\ge \sum_{1\le j \le M} \liminf_{m\to\infty} \ga_{m,j}(\cD_A^c)\\
& \ge \sum_{1\le j \le M} \ga_{j}(\cD_A^c) \ge \xi(\cD_A^c) - \gd.
\ea
\eeq
We finally get our claim by letting $\gd\to 0$.

{\it \noindent Step 2.} It remains to treat the general case, i.e.\ when we make no particular assumption on the number of orbits in $\xi_m$. Let $\gd>0$. By~\cite[Proposition 2.12]{EP23}, the set $\tilde \cM_1^{(2)}$ is dense in $\tilde \cX^{(2)}$ for the metric ${\mathbf D}_2$, so there exists, for every $m\ge 1$, a sequence of probability measures $(\mu_{m,n})_{n\ge 1}$ such that ${\mathbf D}_2(\xi_m, \tilde \mu_{m,n})\to 0$ as $n\to \infty$.  Moreover, using the first step of the proof of~\cite[Proposition 2.12]{EP23}, one may enforce that for every $m\ge 1$,
\beq
\lim_{n\to \infty} \mu_{n,m}(\cD_A^c) = \xi_m(\cD_A^c) (\le \gep_0).
\eeq 
In conclusion, one may find a sequence of probability measures $(\nu_m)_{m\ge1}$ such that $(\tilde \nu_{m})_{m\ge 1}$ converges to $\xi$ as $m\to \infty$ and $\nu_m(\cD_A^c)\le \gep_0 + \gd$. By Step 1 in the present proof, we readily obtain that $\xi(\cD_A^c) \le \gep_0 + \gd$. Finally, we let $\gd\to 0$ to complete the proof.
\end{proof}
\subsection{Lower semi-continuity}
\label{sec:lsc}
The goal of this section is two-fold. Firstly, we provide the minimal continuity properties later required to apply the contraction principle to the relevant functional of the skeleton empirical measure. Here, continuity is meant for the $\bfD$ and $\bfD_2$ topologies, hence extra work is needed compared to~\cite{Phetpradap, BBH2001}. Secondly, we provide a series of approximations that bridge the gap between the functional appearing in Proposition~\ref{pr:app-cond-range} (that is $\phi_{\infty,\gep}$) and the one appearing in Theorem~\ref{thm:2} (that is $\tilde \Gamma$). Adjustments from~\cite{Phetpradap, BBH2001} are required.
\par Let us first extend $\phi_{\infty,\gep}$ to $\tilde \cX^{(2)}$. Since $\phi_{\infty,\gep}$ is well-defined for sub-probability measures on $\bbR^d \times \bbR^d$ and is invariant by the shifts $(\gt_{x,x})_{x\in\R^d}$ (recall the definition from Section~\ref{sec:pairLDP}), we may define for every $\xi = \{\tilde \ga_i\}_{i\in I}\in \tilde\cX^{(2)}$,
\beq
 \tilde \phi_{\infty,\gep}(\xi) := \sum_{i\in I} \phi_{\infty,\gep}(\ga_i) + \kappa \Big(1-\sum_{i\in I} \ga_i(\bbR^d\times\bbR^d)\Big).
\eeq
\begin{lemma}
\label{lem:lsc-phi}
The mapping $\tilde \phi_{\infty,\gep}$ is lower-semi continuous with respect to the metric $\bfD_2$.
\end{lemma}
The proof of Lemma~\ref{lem:lsc-phi}, which is quite long, is deferred to Section~\ref{sec:lem:lsc-phi}.
Let us now define (see (2.89) in~\cite{BBH2001})
\beq
\Psi_{\gep}(\ga) := \int \dd x
\Big(%
1- e^{-\frac{\gk}{\gep}\int_0^\gep \dd s \int p_{ s/d}(x-y)\ga(\dd y)}
\Big),%
\eeq
which is defined for $\ga \in \cM_{\le 1}(\bbR^d)$ and is translation invariant, while for $\xi =\{\tilde\ga_i\}_{i\in I}\in\tilde\cX$, we define
\beq
\tilde\Psi_{\gep}(\xi) = \sum_{i\in I} \Psi_\gep(\ga_i)  + \gk \Big(1- \sum_{i\in I} \ga_i(\bbR^d) \Big)
\,.
\eeq
Recall the definition of $\tilde\Pr$ from Section~\ref{sec:pairLDP}. The following result generalizes~\cite[Lemma~6]{BBH2001}. 
\begin{lemma}
\label{lem:Psieps-to-gamma-tilt}
For all $K>0$,
\beq
\lim_{\gep\to 0}
\sup_{\xi \in \tilde\cX^{(2)}\colon \tilde J^{(2)}_{ \gep/d}(\xi)\le K\gep} 
|\tilde \phi_{\infty,\gep}(\xi)-\tilde\Psi_\gep(\tilde\Pr(\xi))| = 0.
\eeq
\end{lemma}
\begin{proof}[Proof of Lemma~\ref{lem:Psieps-to-gamma-tilt}]
Let $\xi =\{\tilde\ga_i\}_{i\in I}\in\tilde\cX^{(2)}$.
We write $\gs_i := \ga_i(\bbR^d\times \bbR^d) = (\Pr(\ga_i)\otimes \pi_{ \gep/d})(\bbR^d\times \bbR^d)$ and $\bar \ga_i = \ga_i / \gs_i$ for all $i\in I$ and $\gep>0$. Using that $\Pr(\ga_i)(\bbR^d)= \ga_i(\bbR^d \times \bbR^d)$ and adapting the proof of~\cite[Lemma 6]{BBH2001}, we get
\beq
|\tilde \phi_{\infty,\gep}(\xi)-\tilde\Psi_\gep(\tilde\Pr(\xi))| \le \gk \sum_{i\in I} \gs_i \|\bar \ga_i - \Pr(\bar \ga_i)\otimes \pi_{ \gep/d}\|_{\tv}.
\eeq
Moreover, by the standard Pinsker inequality~\cite[Lemma 5(e)]{BBH2001}
\beq
\|\bar \ga_i - \Pr(\bar \ga_i)\otimes \pi_{ \gep/d}\|_{\tv} \le 8 \sqrt{J^{(2)}_{ \gep/d}(\bar \ga_i)} = 8\sqrt{\frac{J^{(2)}_{ \gep/d}(\ga_i)}{\gs_i }}.
\eeq
Recall that $\sum_{i\in I} \gs_i \le 1$. By the Cauchy-Schwarz inequality, we obtain
\beq
|\tilde \phi_{\infty,\gep}(\xi)-\tilde\Psi_\gep(\tilde\Pr(\xi))| \le 8\gk \sqrt{\sum_{i\in I} J^{(2)}_{ \gep/d}(\ga_i)} = 8\gk \sqrt{\tilde J^{(2)}_{ \gep/d}(\mu)},
\eeq
which completes the proof.
\end{proof}
Recall the definition of $\tilde\Gamma$ in~\eqref{eq:def-tilde-JG}. We define for any $\xi=\{\tilde\alpha_i\}_{i\in I}\in\tilde\cX$ and $\delta>0$,

\beq
\tilde \Gamma_\delta(\xi):= \sum_{i\in I} \int\dd x \Big(1-e^{-\gk (p_{ \delta/d}* \alpha_i)(x)}\Big) + \gk\Big(1- \sum_{i\in I} \ga_i(\bbR^d) \Big).
\eeq

\begin{lemma}\label{lem:Psieps-to-gamma} For all $K>0$,
	\beq
	\label{eq:Psieps-to-gamma}
	\lim_{\gep \to 0} \sup_{\xi\in \tilde\cX \colon \tilde J_{ \gep/d}(\xi) \le K\gep}
	\Big|\tilde\Gamma(\xi) - \tilde\Psi_\gep(\xi)\Big| =0.
	\eeq
	Moreover, there exists $C\in(0,\infty)$ such that for any $\delta\geq \gep$,
\beq
\label{eq:Psieps-to-gamma2}
\sup_{\xi\in\tilde\cX \colon \tilde{J}_{ \gep/d}(\xi)\leq K\gep}\Big| \tilde\Gamma(\xi)-\tilde\Gamma_\delta(\xi)\Big|\le C\sqrt{\delta}\,.
\eeq	
\end{lemma}
The proof is deferred to Section~\ref{S:Gamma}. We conclude this section with the following:
\begin{proposition}\label{prop:restGamma-lsc} 
For all $K>0$, the restriction of $\tilde \Gamma$ to $\cA(K) := \{\xi \in \tilde \cX\colon \tilde J(\xi) \le K\}$ is lower semi-continuous. Moreover, for all $\delta>0$, the map $\tilde\Gamma_\delta$ is lower semi-continuous on $\tilde\cX$.
\end{proposition}
The proof is deferred to Section~\ref{S:gammalsc}.
\subsection{Conclusion: Proof of Proposition~\ref{pr:LDP-measure}}
\label{sec:conclusion-proof}
Along this section, we shall say that a real-valued sequence $(u_n)$ is \emph{negligible} if for some $c > \frac1d I(b)$ (see Theorem~\ref{thm:MDP-rw}) we have $|u_n| \le \exp(-cn^{1-2/d})$ for all $n$ large enough.
Let $\gep_0>0$. By Proposition~\ref{pr:conc.ineq},
\beq
\bP(R_n\le bn, \tilde L_{M,\gep}\in F) \le
\bP(\bE(R_n | \hat S)\le (b+\gep_0)n, \tilde L_{M,\gep}\in F) +O(e^{-C(\gep,\gep_0)n^{1-2/d}}),
\eeq
where $C(\gep,\gep_0)$ goes to infinity as $\gep$ goes to zero. Hence the second term is \emph{negligible} provided $\gep$ is chosen small enough.
By Proposition~\ref{pr:app-cond-range}, we obtain for all $A>0$ and $n\ge n_0(\gep,\gep_0,A)$,
\beq
\bP(\bE(R_n | \hat S)\le (b+\gep_0)n, \tilde L_{M,\gep}\in F) \le (a) + (b),
\eeq
where
\beq
\ba
{\rm (a)} &= \bP(\phi_{\infty,\gep}(L^{(2)}_{M,\gep})\le b+7\gep_0,\ \tilde L_{M,\gep}\in F,\ \tilde L_{M,\gep}^{(2)}\in \cM(A,\gep_0)),\\
{\rm (b) } &= \bP(L^{(2)}_{M,\gep}\notin \cM(A,\gep_0)).
\ea
\eeq
By Proposition~\ref{pr:restricy-PEM} the term (b) is \emph{negligible} upon choosing $A$ suitable. Let us now focus on (a), which is the main term and which we may write as
\beq
\ba
{\rm (a)}= \bP(\tilde\phi_{\infty,\gep}(\tilde L^{(2)}_{M,\gep})\le b+ 7\gep_0,\ \tilde\Pr(\tilde L^{(2)}_{M,\gep})\in F,\ \tilde L_{M,\gep}^{(2)}\in \cM(A,\gep_0))\,.
\ea
\eeq
By Lemma~\ref{lem:Projection} the last term can be bounded from above by
\beq
\bP(\tilde\phi_{\infty,\gep}(\tilde L^{(2)}_{M,\gep})\le b+7\gep_0,\ \tilde{\Pr}_A(\tilde L^{(2)}_{M,\gep})\in F_{2\gep_0},\ \tilde L_{M,\gep}^{(2)}\in \cM(A,\gep_0))\,,
\eeq
where for any $a\geq 0$,
\beq
F_{a}=\{\xi\in\tilde\cX\,:\, \bfD(\xi, F)\leq a\}\,,
\eeq
which itself is a closed set.
Since $\tilde{\Pr}_A$ is continuous, using Lemmas~\ref{lem:Projection} and~\ref{lem:lsc-phi}, we see that the set $\tilde{\Pr}_A^{-1}(F_{2\gep_0})\cap \tilde\phi_{\infty,\gep}^{-1}(-\infty,b+7\gep_0]\cap\cM(A,\gep_0)$ is closed for the $\bfD_2$-topology. Thus, combining this with Proposition~\ref{pr:LDPpair},
we get
\beq
\ba
\label{eq:UB-LDPpair}
\limsup n^{2/d -1}\log \bP(\tilde\phi_{\infty,\gep}(\tilde L^{(2)}_{M,\gep})\le b+ 7\gep_0,\ &\tilde{\Pr}_A(\tilde L^{(2)}_{M,\gep})\in F_{2\gep_0} , \tilde L_{M,\gep}^{(2)}\in \cM(A,\gep_0))\\ &\le -
\inf_{\xi\in\tilde{\Pr}_A^{-1}(F_{2\gep_0})\cap \tilde\phi_{\infty,\gep}^{-1}(-\infty,b+7\gep_0]\cap\cM(A,\gep_0)} \frac{1}{\gep} \tilde J^{(2)}_{ \gep/d}(\xi)\,.
\ea
\eeq
Define
\beq
\cG_\gep(b):=\{\xi \in \tilde\cX:\, \tilde\Psi_\gep(\xi)\leq b\}.
\eeq
Now, we argue that for $\gep$ sufficiently small, by Lemma~\ref{lem:Psieps-to-gamma-tilt},
\beq
\ba
\inf_{\xi\in\tilde{\Pr}_A^{-1}(F_{2\gep_0})\cap \tilde\phi_{\infty,\gep}^{-1}(-\infty,b+7\gep_0]\cap\cM(A,\gep_0)} \frac{1}{\gep} \tilde J^{(2)}_{ \gep/d}(\xi) &\ge \inf_{\xi\in\tilde{\Pr}_A^{-1}(F_{2\gep_0})\cap\tilde\Pr^{-1}(\cG_\gep(b+8\gep_0))\cap\cM(A,\gep_0)}\frac{1}{\gep}\tilde J^{(2)}_{ \gep/d}(\xi)\,.
\ea
\eeq
To continue we invoke again Lemma~\ref{lem:Projection} which shows that the last term above is bounded from below by
\beq
\inf_{\xi\in\tilde{\Pr}^{-1}(F_{4\gep_0})\cap\tilde\Pr^{-1}(\cG_\gep(b+8\gep_0))}\frac{1}{\gep}\tilde J^{(2)}_{ \gep/d}(\xi)
= \inf_{\xi\in F_{4\gep_0}\cap\cG_\gep(b+8\gep_0)}\frac{1}{\gep}\tilde J_{ \gep/d}(\xi)\,,
\eeq
where the equality follows from~\eqref{eq:contractionP}.
Summing up, we have proven so far that for all $\gep$ sufficiently small
\beq\label{eq:summary}
\limsup_{n\to\infty} n^{\frac2d -1} \log \bP(R_n\le bn, \tilde L_{M,\gep}\in F) \le -\inf_{\xi\in F_{4\gep_0}\cap\cG_\gep(b+8\gep_0)}\frac{1}{\gep}\tilde J_{\gep/d}(\xi).
\eeq
We now investigate the limit of the right-hand side as $\gep\to 0$. To that end we first note that if
\beq\label{eq:liminfinite}
\liminf_{\gep\to 0}\inf_{\xi\in F_{4\gep_0}\cap\cG_\gep(b+8\gep_0)}\frac{1}{\gep}\tilde J_{ \gep/d}(\xi) = \infty\,,
\eeq
then we can immediately conclude the result. Hence, we can assume that the above limit is finite, and therefore that at least a long a sub-sequence $(\gep_n)_{n\in\N}$ there is a constant $K>0$ such that for all $n\in\N$
\beq\label{eq:limfinite}
\inf_{\xi\in F_{4\gep_0}\cap\cG_{\gep_n}(b+8\gep_0)}\frac{1}{\gep_n}\tilde J_{ \gep_n/d}(\xi)\leq K\,.
\eeq
In the sequel we will suppress the sub-sequence from the notation.
To continue we will need the following result, whose proof is deferred to Section~\ref{S:LB-JepsK}:
	\begin{proposition}
	\label{pr:LB-JepsK}
	For any compact set $\cK\subseteq \tilde \cX$ (in the $\bfD$-topology),
	\beq
	\label{eq:LB-JepsK}
	\liminf_{\gep \to 0} \inf_{\xi \in \cK}
	\frac 1\gep \tilde J_{\gep/d}(\xi) \ge 
{ \frac{1}{d}}	\inf_{\xi \in \cK} \tilde J(\xi).
	\eeq
\end{proposition}
We now finish the proof of Proposition~\ref{pr:LDP-measure}. Define
\beq
\tilde{\cD}_\delta(b):=\{\xi\in \tilde\cX:\, \tilde\Gamma_\delta(\xi)\leq b\}\,,
\eeq
and recall the definition of $\tilde\cD(b)$ in the formulation of Proposition~\ref{pr:LDP-measure}.
By Lemma~\ref{lem:Psieps-to-gamma}, and the observation made around Equation~\eqref{eq:limfinite} we can write for any $\gep$ sufficiently small, and $\delta\geq 4\gep_0$, and some fixed constant $C>0$,
\beq
\ba
-\inf_{\xi\in F_{4\gep_0}\cap\cG_\gep(b+8\gep_0)}\frac{1}{\gep}\tilde J_{\gep/d}(\xi)&=-\inf_{\substack{\xi\in F_{4\gep_0}\cap\cG_\gep(b+8\gep_0),\\ \tilde{J}_{ \gep/d}(\xi)\leq K\gep}}\frac{1}{\gep}\tilde J_{ \gep/d}(\xi)\\
&\leq -\inf_{\substack{\xi\in F_{4\gep_0}\cap\tilde{\cD}(b+9\gep_0),\\\tilde{J}_{ \gep/d}(\xi)
\leq K\gep}}\frac{1}{\gep}\tilde {J}_{ \gep/d}(\xi)\\
&\leq -\inf_{\substack{\xi\in F_{\delta}\cap\tilde{\cD}_\delta(b+C\sqrt{\delta}),\\ \tilde{J}_{ \gep/d}(\xi)\leq K\gep}}\frac{1}{\gep}\tilde {J}_{ \gep/d}(\xi)\\
&\leq  -\inf_{\xi\in F_{\delta}\cap\tilde{\cD}_\delta(b+C\sqrt{\delta})}\frac{1}{\gep}\tilde {J}_{ \gep/d}(\xi)\,.
\ea
\eeq
Note that $F_{\delta}\cap \tilde{\cD}_\delta(b+C\sqrt{\delta})\subseteq \tilde\cX$ is closed by the second part of Proposition~\ref{prop:restGamma-lsc}, thus compact. Hence, Proposition~\ref{pr:LB-JepsK} implies that
\beq
-\liminf_{\gep\to 0} \inf_{\xi\in F_{\delta}\cap\tilde{\cD}_\delta(b+C\sqrt{\delta})}\frac{1}{\gep}\tilde {J}_{ \gep/d}(\xi) \leq - { \tfrac 1d}\inf_{\xi\in F_{\delta}\cap\tilde{\cD}_\delta(b+C\sqrt{\delta})}\tilde {J}(\xi).
\eeq
It remains to send $\delta\to 0$. To that end, we note that using the same arguments as in Equation~\eqref{eq:limfinite} we can again assume that there is some $K>0$ such that at least along a subsequence of $\delta$'s converging to zero we have that
\beq
\inf_{\xi\in F_{\delta}\cap\tilde{\cD}_\delta(b+C\sqrt{\delta})}\tilde {J}(\xi)\leq K
\eeq
for all such $\delta$. We will again suppress the choice of subsequence from the notation.
Thus, recalling that $\cA(K) := \{\xi \in \tilde \cX\colon \tilde J(\xi) \le K\}$, we can write 
\beq
-\inf_{\xi\in F_{\delta}\cap\tilde{\cD}_\delta(b+C\sqrt{\delta})}\tilde {J}(\xi) = -\inf_{\xi\in F_{\delta}\cap\tilde{\cD}_\delta(b+C\sqrt{\delta})\cap\cA(K)}\tilde {J}(\xi) \leq -\inf_{\xi\in F_{\delta}\cap\tilde{\cD}(b+2C\sqrt{\delta})\cap \cA(K)}\tilde {J}(\xi)\,,
\eeq
where we used Lemma~\ref{lem:Psieps-to-gamma} to obtain the last inequality. Letting $\gd\to 0$ with the help of Lemma~\ref{lem:deltatozero} below and using that
\beq
\inf_{\xi\in F\cap \tilde{\cD}(b)\cap \cA(K)} \tilde{J}(\xi) \ge
\inf_{\xi\in F\cap \tilde{\cD}(b)} \tilde{J}(\xi),
\eeq
allows to conclude the proof of Proposition~\ref{pr:LDP-measure}.
\begin{lemma}\label{lem:deltatozero}
	Let $F \subseteq\tilde{\cX}$ be a closed set. Then,
	\begin{equation}\label{eq:deltatozero}
	\lim_{\delta\to 0}\inf_{\xi\in F_{\delta}\cap \tilde{\cD}(b+\gd)\cap\cA(K)} \tilde{J}(\xi) = \inf_{\xi\in F\cap \tilde{\cD}(b)\cap \cA(K)} \tilde{J}(\xi)\,.
	\end{equation}
\end{lemma}

	\begin{proof}[Proof of Lemma~\ref{lem:deltatozero}] Since $\tilde{\cD}(b)\subseteq \tilde{\cD}(b+\gd)$, we directly have
	
\beq
	\limsup_{\delta\to 0}\inf_{F_{\delta}\cap \tilde{\cD}(b+\gd)\cap\cA(K)} \tilde{J} \le \inf_{F\cap \tilde{\cD}(b)\cap \cA(K)} \tilde{J}.
\eeq
From now on, we focus on the reversed inequality.
By Proposition~\ref{prop:restGamma-lsc} (first part of the statement) and the compactness of $\tilde{\cX}$, the set $F_{\delta}\cap\tilde\cD(b+\gd)\cap\cA(K)$ is actually compact. Hence, there exists a sequence $(\xi_\delta)_{\delta>0}$ of minimizers of the left hand side in~\eqref{eq:deltatozero}. By the compactness of $\tilde{\cX}$ we may extract a subsequence converging to some $\xi_0\in\tilde{\cX}$, which for ease of notation we again denote by $(\xi_\delta)_{\delta >0}$. By the lower semi-continuity of $\tilde J$, we get that $\xi_0\in \cA(K)$ and by the lower semi-continuity of $\tilde\Gamma$ restricted to $\cA(K)$, we get that $\tilde\Gamma(\xi_0) \le \liminf_{\gd\to 0} \tilde\Gamma(\xi_\gd) \le b$, hence $\xi_0\in \tilde \cD(b)$. It is moreover not difficult to see that we also have that $\xi_0\in F$. In conclusion,
\beq
\liminf_{\delta\to 0}\inf_{F_\delta\cap \tilde{\cD}(b+\gd)\cap\cA(K)} \tilde{J}
= \liminf_{\delta\to 0} \tilde{J}(\xi_\gd) \ge \tilde J(\xi_0) \ge \inf_{F\cap \tilde\cD(b)\cap \cA(K)} \tilde J.
\eeq
\end{proof}
\subsection{Proof of Lemma~\ref{lem:lsc-phi}}
\label{sec:lem:lsc-phi}
We proceed in several steps.\\
{\noindent \it (i) Truncation procedure.}
For $\eta\in(0,\gep/2)$ and $\mu\in\cM_{\le 1}(\bbR^d\times\bbR^d)$, let us define
\beq
\label{eq:truncatedphi}
\phi_{\infty,\gep,\eta}(\mu) := \int_{\bbR^d} \dd z
	\Big(%
	1- e^{-\frac{\gk}{\gep}\int \varphi_{\gep,\eta}(z-x,z-y)\mu(\dd x,\dd y)}
	\Big)%
\eeq
where
\beq
\varphi_{\gep,\eta}(x,y) := \int_{\eta}^{\gep-\eta} \frac{p_{ s/d}(-x)p_{ (\gep-s)/d}(y)}{p_{ \gep/d}(x-y)} \dd s,
\eeq
and for $\xi = \{\tilde \alpha_i\}_{i\in I}$,
\beq
\tilde\phi_{\infty,\gep,\eta}(\xi):=\sum_{i\in I} \phi_{\infty,\gep,\eta}(\alpha_i) + \gk\Big(1-\frac{2\eta}{\gep}\Big)\Big(1-\sum_{i\in I} \ga_i(\bbR^d\times\bbR^d)\Big).
\eeq
It is enough to show that $\tilde\phi_{\infty,\gep,\eta}$ is lower semi-continuous with respect to the metric $\mathbf{D}_2$, since $\tilde\phi_{\infty,\gep}$ is the supremum of $\tilde\phi_{\infty,\gep,\eta}$ with respect to $\eta$.\\
{\noindent \it (ii) Rewriting the truncated function.} 
By expanding the exponential in~\eqref{eq:truncatedphi}, and using Lemma~\ref{lem:phi.q.prob} we get for all $\mu\in\cM_{\le 1}(\bbR^d\times \bbR^d)$,
\beq
\ba
&\tilde\phi_{\infty,\gep,\eta}(\{\mu\}) - \gk(1-\tfrac{2\eta}{\gep})\\
&\qquad =
\int_{\bbR^d}\dd z \sum_{n\ge 2} \frac{(-1)^{n+1}}{n!} \Big(\frac{\gk}{\gep}\Big)^n \int_{(\bbR^d\times\bbR^d)^n} \prod_{i=1}^n \varphi_{\gep,\eta}(z-x_i,z-y_i) \mu(\dd x_i, \dd y_i).
\ea
\eeq
By Fubini-Tonelli's theorem, we may interchange the sum over $n$ and the integral over $z$. Indeed,
\beq\label{eq:FubiniTonelli}
\ba
&\int_{\bbR^d}\dd z \sum_{n\ge 2}  \frac{1}{n!} \Big(\frac{\gk}{\gep}\Big)^n \int_{(\bbR^d\times\bbR^d)^n} \prod_{i=1}^n \varphi_{\gep,\eta}(z-x_i,z-y_i) \mu(\dd x_i, \dd y_i)\\
&= \int_{\bbR^d}\dd z \sum_{n\ge 2} \frac{1}{n!} \Big(\frac{\gk}{\gep}\Big)^n\Big(\int_{\bbR^d\times\bbR^d}  \varphi_{\gep,\eta}(z-x,z-y) \mu(\dd x, \dd y)\Big)^n\\
&\le \int_{\bbR^d}\dd z \Big(e^{\tfrac{\gk}{\gep} \int_{\bbR^d\times\bbR^d}  \varphi_{\gep,\eta}(z-x,z-y) \mu(\dd x, \dd y)} - 1\Big)\\
&\le e^{\tfrac{\gk}{\gep}\|\varphi_{\gep,\eta}\|_\infty} 
\int_{\bbR^d} \int_{\bbR^d\times\bbR^d} \frac{\gk}{\gep} \varphi_{\gep,\eta}(z-x,z-y) \mu(\dd x, \dd y) \dd z\\
&\le \gk e^{\tfrac{\gk}{\gep}\|\varphi_{\gep,\eta}\|_\infty}\mu(\R^d\times\R^d) < \infty.
\ea
\eeq
We have used Lemma~\ref{lem:phi.q.prob} and the comment slightly below Lemma~\ref{lem:uniformvarphi} on the boundedness of the function $\varphi_{\gep,\eta}$.
We obtain thereby:
\beq
\ba
\label{eq:phimu-exp}
&\tilde\phi_{\infty,\gep,\eta}(\{\mu\}) -\gk(1-\tfrac{2\eta}{\gep}) \\
&\qquad =
\sum_{n\ge 2} \frac{(-1)^{n+1}}{n!} \Big(\frac{\gk}{\gep}\Big)^n \int_{(\bbR^d\times\bbR^d)^n} V_n^{(\eta)}(x_1,y_1,\ldots, x_n, y_n) \prod_{i=1}^n\mu(\dd x_i, \dd y_i),
\ea
\eeq
where
\beq
V_n^{(\eta)}(x_1,y_1,\ldots, x_n, y_n) := \int_{\bbR^d}\dd z \prod_{i=1}^n \varphi_{\gep,\eta}(z-x_i,z-y_i).
\eeq
{\noindent \it (iii) Lower-semi continuity of the truncated function.} The $V_n$'s defined above (we omit $\eta$) are clearly \emph{translation invariant}, and we will also prove that they are \emph{continuous} (Step 1 below) and \emph{vanishing at infinity} (Step 2 below). It then follows by the definition of the metric $\mathbf{D}_2$ that
\beq
\xi\mapsto \sum_{i\in I} \int V_n^{(\eta)}(x_1,y_1,\ldots,x_n,y_n)\prod_{j=1}^{n} \alpha_i(\dd x_j, \dd y_j)
\eeq
is continuous. Moreover, by~\eqref{eq:FubiniTonelli},
\beq
\ba
\sum_{i\in I}\sum_{n\geq 2} &  \frac{1}{n!} \Big(\frac{\gk}{\gep}\Big)^n
\int_{(\bbR^d\times\bbR^d)^n} V_n^{(\eta)}(x_1,y_1,\ldots x_n, y_n) \prod_{j=1}^n\alpha_i(\dd x_j, \dd y_j)\\
&\leq \sum_{i\in I} \gk e^{\tfrac{\gk}{\gep}\|\varphi_{\gep,\eta}\|_\infty}\alpha_i(\R^d\times\R^d) <\infty.
\ea
\eeq
This allows us to use Fubini-Tonelli's theorem and obtain
\beq
\ba
\sum_{i\in I} \sum_{n\ge 2} & \frac{(-1)^{n+1}}{n!} \Big(\frac{\gk}{\gep}\Big)^n \int_{(\bbR^d\times\bbR^d)^n} V_n^{(\eta)}(x_1,y_1,\ldots x_n, y_n) \prod_{j=1}^n\alpha_i(\dd x_j, \dd y_j)\\
&= \sum_{n\ge 2}  \frac{(-1)^{n+1}}{n!} \Big(\frac{\gk}{\gep}\Big)^n \sum_{i\in I} \int_{(\bbR^d\times\bbR^d)^n} V_n^{(\eta)}(x_1,y_1,\ldots x_n, y_n) \prod_{j=1}^n\alpha_i(\dd x_j, \dd y_j),
\ea
\eeq
with the sum over $n$ on the right hand side converging absolutely. Therefore, the right-hand side is continuous with respect to the metric $\mathbf{D}_2$.\\ 

{\noindent \it (iv) Conclusion.}
As announced in Step {\it (iii)} above, it remains to prove the two required properties of the function $V_n^{(\eta)}$ in order to complete this section.\\
{\bf \noindent Step 1}. {\it Continuity of $V_n^{(\eta)}$.} First, one can prove that for all $z$, the function 
\beq
(x_1,y_1,\ldots x_n, y_n) \mapsto \prod_{i=1}^n \varphi_{\gep,\eta}(z-x_i,z-y_i)
\eeq
is continuous (by dominated convergence for instance). Now, observe that for any $(x,y)$ in a compact set $K\subseteq \bbR^d\times \bbR^d$ (say $|x|,|y|\le M$) and $\eta>0$
\beq
\varphi_{\gep,\eta}(z-x,z-y) \le C(\gep,\eta,M)
\exp\Big(-\frac{2d|z|^2}{\gep}+ \frac{\gep{ d} M|z|}{\eta(\gep-\eta)}\Big),
\eeq
which is integrable in $z$, hence the continuity of $V_n^{(\eta)}$, by dominated convergence. To get the inequality above, use that
\beq
\ba
\frac{|x-y|^2}{2\gep} - \frac{|z-x|^2}{2s} - \frac{|z-y|^2}{2(\gep-s)} &= 
\frac{|x-y|^2}{2\gep} - \frac{|x|^2}{2s} - \frac{|y|^2}{2(\gep-s)} - \frac{\gep|z|^2}{2s(\gep-s)}\\
&\qquad + \frac{<z,x>}{s} + \frac{<z,y>}{(\gep-s)}\\
& \le - \frac{\gep|z|^2}{2s(\gep-s)} + \frac{|x||y|}{\gep}+ \frac{|z||x|}{s} + \frac{|z||y|}{(\gep-s)}.
\ea
\eeq

{\bf \noindent Step 2.} {\it Evanescence of $V_n^{(\eta)}$ at infinity.} Let us start with an estimate:
\begin{lemma}\label{lem:uniformvarphi} For all $x,y\in\bbR^d$, $\eta\in (0,\gep/2)$,
\beq
\varphi_{\gep,\eta}(x,y) \le C(\eta,\gep)\min\Big(1,(\sqrt{|x|}+\sqrt{|y|})^{-2}\Big)
\exp\Big[\tfrac{ d}{2\gep}\Big(|x-y|^2 - (|x|+|y|)^2\Big)\Big].
\eeq
\end{lemma}
Note that the term in the exponential is nonpositive, and equals zero if and only if $x$ and $y$ are colinear with opposite directions (i.e $\langle x,y\rangle = -|x||y|$). In particular, the function $\varphi_{\gep,\eta}$ is bounded.
\begin{proof}[Proof of Lemma~\ref{lem:uniformvarphi}]
For all $s\in (0,\gep)$, $\ga,\gb>0$, define
\beq
f_{\ga,\gb}(s) = \frac{\ga}{s} + \frac{\gb}{\gep-s}.
\eeq
By computing the first four derivatives, we see that the function achieves its minimum $\tfrac{1}{\gep}(\ga^{1/2}+ \gb^{1/2})^2$ at $\frac{\gep\ga^{1/2}}{\ga^{1/2}+ \gb^{1/2}}$ and that its second derivative achieves its minimum $\tfrac{2}{\gep^3}(\ga^{1/4}~+~ \gb^{1/4})^4$ at $\frac{\gep\ga^{1/4}}{\ga^{1/4}+ \gb^{1/4}}$. Therefore, for all $s\in(0,\gep)$,  we get by using a second-order Taylor expansion:
\beq
f_{\ga,\gb}(s) - \frac{1}{\gep}(\ga^{1/2}+ \gb^{1/2})^2 \ge  \frac{1}{\gep^3}\Big(\ga^{1/4}+ \gb^{1/4} \Big)^4 \Big(s - \frac{\gep\ga^{1/2}}{\ga^{1/2}+ \gb^{1/2}}\Big)^2.
\eeq
Applying this inequality to $\ga = |x|^2$ and $\gb = |y|^2$, we obtain
\beq
\ba
\varphi_{\gep,\eta}(x,y)
 \le \frac{\gep^{d/2}}{[2\pi\eta^2]^{d/2}} &\int_0^\gep\exp\Big[-\frac{ d}{2\gep^3}(s - \tfrac{\gep|x|}{|x|+|y|})^2(\sqrt{|x|}+\sqrt{|y|})^4\Big]\dd s\\
& \times\exp\Big[\frac{ d}{2\gep}(|x-y|^2-(|x|+|y|)^2)\Big]\,.
\ea
\eeq
Using that $\int_{\bbR} e^{-cs^2}\dd s = \sqrt{\pi/c}$, we get the result.
\end{proof}
Let us come back to Step 2. Let $n\ge 2$ and $\eta\in (0, \gep/2)$. We now show that $V_n^{(\eta)}$ is vanishing. To this end, consider a sequence
\beq
(x_{1,N},y_{1,N},x_{2,N},y_{2,N},\ldots, x_{n,N},y_{n,N})
\eeq
such that 
\begin{equation}\label{eq:maxinf}
	\max_{1\leq i \le j\leq n}\big\{|x_{j,N}-x_{i,N}|\vee |y_{j,N}-y_{i,N}|\vee |x_{j,N}-y_{i,N}|\big\}\ \stackrel{N\to\infty}{\longrightarrow}\infty\,.
\end{equation}
In the sequel we suppress the $N$ from the notation and assume w.l.o.g.\ that maximizing indexes $i$ and $j$ can be found in $\{1,2\}$. Pick $M>0$. For $N$ large enough,  we then have
\beq
\max(|y_1- y_2|,|x_1-x_2|,|x_1-y_1|,|x_2-y_2|, |y_1-x_2|, |x_1-y_2|)\ge M.
\eeq
If $|y_1-y_2|\ge M$, we may write
\beq
\bbR^d = A_1\cup A_2, \qquad \textrm{where\ }A_i := \{z\in \bbR^d \colon |z-y_i|\ge M/2\}, \qquad i\in\{1,2\}.
\eeq
Using the boundedness of $\varphi_{\gep,\eta}$ (Lemma~\ref{lem:uniformvarphi}) we see that it is sufficient to focus on the product of $\varphi_{\gep,\eta}(\cdot -x_i,\cdot-y_i)$ for $i=1,2$. Using Lemma~\ref{lem:uniformvarphi} and the fact that $z\mapsto\varphi_{\gep,\eta}(z-x_2, z-y_2)$ is a sub-probability density, we get
\beq\label{eq:A1plus}
\int_{A_1} \dd z \prod_{i=1}^{2}\varphi_{\gep,\eta}(z-x_i,z-y_i)
\le {\rm (cst)} \int_{A_1}\dd z \frac{1}{|z-y_1|}\varphi_{\gep,\eta}(z-x_2,z-y_2)
\leq \frac{{\rm (cst)}}{M},
\eeq
and similarly for the integral over $A_2$.
The other cases, namely $|x_1-x_2|\ge M$, $|x_1-y_1|\ge M$, $|x_2-y_2|\ge M$, $|y_1-x_2|\ge M$ and $|x_1-y_2|\ge M$ can be dealt with in the same way.

 \subsection{Proof of Lemma~\ref{lem:Psieps-to-gamma}}\label{S:Gamma}
In this section, we prove Lemma~\ref{lem:Psieps-to-gamma} and take the occasion to correct a glitch in~\cite{BBH2001}, see Remark~\ref{rmk:fixlemmaBBH} below.
We proceed in two steps. For notational convenience we will bound the difference between $\tilde\Gamma$ and $\tilde\Psi_{ d\gep}$ and between $\tilde\Gamma$ and $\tilde\Gamma_\delta$ in terms of $\tilde J_\gep$. This is of course equivalent to the statement of the result.\\
{\bf \noindent Step 1. Proof of~\eqref{eq:Psieps-to-gamma}.} We further split this step in two parts. We first treat the case of $\xi$ consisting of a single orbit and then treat the general case.\\
{\it \noindent (i) Single orbit case.}
Let $\nu\in\cM_{\leq 1}(\bbR^d)$ and $f$ be its density with respect to the Lebesgue measure. Assume $\xi = \{\tilde\nu\}$. Following (2.94) and (2.95) in the proof of~\cite[Lemma 7]{BBH2001}, we obtain:
\beq
|\tilde\Gamma(\xi) - \tilde\Psi_{ d\gep}(\xi)| \le \frac{\gk}{\gep} \int_0^\gep \dd s \|\nu \pi_s - \nu\|_{\tv},
\eeq
where
\beq
\|\nu \pi_s - \nu\|_{\tv} := \int_{\bbR^d} |f(x) - \pi_s f(x)|\dd x.
\eeq
The integrand is split and bounded as follows:
\beq
\|\nu \pi_s - \nu\|_{\tv} \le \|\nu \pi_s - \nu \pi_{s+\gep}\|_{\tv} + \|\nu \pi_{s+\gep} - \nu\|_{\tv}.
\eeq
As in~\cite[Lemma 7]{BBH2001}, the first term in the sum is bounded by $8\sqrt{J_\gep(\nu)}$. We now bound the second term. From this point on, the proof differs from~\cite[Lemma 7]{BBH2001}, see Remark~\ref{rmk:fixlemmaBBH} below. Writing 
$f - \pi_{s+\gep} f = (\sqrt{f}-\sqrt{\pi_{s+\gep} f})(\sqrt{f}+\sqrt{\pi_{s+\gep} f})$ and using the Cauchy-Schwartz inequality, we obtain
\beq
\|\nu \pi_{s+\gep} - \nu\|_{\tv} = \int_{\bbR^d} |f(x) - \pi_{s+\gep} f(x)| \dd x \le 2 H(\nu, \nu \pi_{s+\gep}),
\eeq
where 
\beq
H(\nu, \nu \pi_{s+\gep}) := \Big(\int_{\bbR^d} (\sqrt{f(x)} - \sqrt{\pi_{s+\gep} f(x)})^2\dd x \Big)^{1/2}
\eeq
is the Hellinger distance between $\nu$ and $\nu\pi_{s+\gep}$. Note that $\sqrt{\pi_{s+\gep} f} \ge \pi_{s+\gep} \sqrt{f}$, by Jensen's inequality. Therefore, using that $\nu$ and $\nu\pi_{s+\gep}$ have the same mass,
\beq
H(\nu,\nu \pi_{s+\gep})^2 = 2 \Big(\int f(x)  - \int \sqrt{f(x)} \sqrt{\pi_{s+\gep} f(x)}\Big) \le 2 L_{s+\gep}(\sqrt{f}),
\eeq
where we have defined
\beq
L_t(g) := \bra g, g \ket - \bra g, \pi_t g \ket, \qquad g\in L^2(\bbR^d), \qquad t\ge 0,
\eeq
$\bra\,\cdot,\cdot \ket$ being the usual inner product on the space of square integrable functions. By~\cite[Theorem 7.10]{LiebLoss}, we know that for all $g\in L^2(\bbR^d)$, the map $t>0\mapsto L_t(g)/t$ is non-increasing (monotonicity is actually hidden in the proof of that theorem). Finally, note that for any $t>0$
\beq
J_t(\nu) = \sup_{u>0} - \int \log\Big(\frac{\pi_t u(x)}{u(x)}\Big) f(x) \dd x \ge 
- \int \log\Big(\frac{\pi_t \sqrt{f(x)}}{\sqrt{f(x)}}\Big) f(x) \dd x \ge L_t(\sqrt{f}),
\eeq
using $\log(1+z)\le z$ for the last inequality. Summing up, we obtain for $0\le s\le \gep$,
\beq
\ba
\|\nu \pi_{s+\gep} - \nu\|^2_{\tv} \le 4 H(\nu, \nu\pi_{s+\gep})^2
\le 8 L_{s+\gep}(\sqrt{f})
&= 8(s+\gep) \frac{L_{s+\gep}(\sqrt{f})}{s+\gep}\\
&\le 8(s+\gep) \frac{L_{\gep}(\sqrt{f})}{\gep}\\
&\le 16 L_{\gep}(\sqrt{f})\\
&\le 16 J_\gep(\nu),
\ea
\eeq
and finally, 
\beq
|\Gamma(\nu) - \Psi_{ d\gep}(\nu)| \le 12\gk\sqrt{J_\gep(\nu)}.
\eeq
{\it \noindent (ii) General case.} Assume $\xi = \{\tilde\ga_i, i\in I\}$. Similarly as in Step (i) and using the triangular inequality, we get:
\beq
|\tilde\Gamma(\xi) - \tilde\Psi_{ d\gep}(\xi)| \le  \sum_{i\in I} \frac{\gk}{\gep}\int_0^\gep \dd s \|\ga_i \pi_s - \ga_i\|_{\tv}.
\eeq
Define $\gs_i := \ga_i(\bbR^d)$ and $\bar \ga_i := \ga_i/\gs_i \in\cM_1(\bbR^d)$ for all $i\in I$. From what precedes in (i),
\beq
\|\ga_i \pi_s - \ga_i\|_{\tv} = \gs_i \|\bar\ga_i \pi_s - \bar\ga_i\|_{\tv} \le 12\gs_i\sqrt{J_\gep(\bar\ga_i)} = 12\sqrt{\gs_i}\sqrt{J_\gep(\ga_i)}.
\eeq
We may now conclude with the Cauchy-Schwarz inequality, since $\sum \gs_i \le 1$, that
\beq
|\tilde\Gamma(\xi) - \tilde\Psi_{ d\gep}(\xi)| \le 12\gk\sqrt{J_\gep(\xi)},
\eeq
which completes the proof.\\
{\bf \noindent Step 2. Proof of~\eqref{eq:Psieps-to-gamma2}} We re-use the arguments from the previous step.
For simplicity, let us stick to the case of a single orbit, i.e. $\xi=\{\tilde\nu\}$, where $\nu\in\cM_{\le 1}(\bbR^d)$ with a density $f$ w.r.t.\ Lebesgue measure. Then,
\beq
|\tilde\Gamma(\xi) - \tilde\Gamma_{ d\gd}(\xi)| \le \gk\|\nu\pi_\gd - \nu\|_{\tv} \le 4\gk\sqrt{L_\gd(\sqrt{f})}.
\eeq
Since $\gd \ge \gep$,
\beq
\frac{L_\gd(\sqrt{f})}{\gd}\le  \frac{L_\gep(\sqrt{f})}{\gep}\le  \frac{J_\gep(\xi)}{\gep},
\eeq
which concludes the proof.
\begin{remark}\label{rmk:fixlemmaBBH}
Equation~\eqref{eq:Psieps-to-gamma} in Lemma~\ref{lem:Psieps-to-gamma} extends~\cite[Lemma 7]{BBH2001} to the space $\tilde{\cX}$. Our proof actually corrects a flaw in the proof of~\cite[Lemma 7]{BBH2001}. The latter proof uses~\cite[Lemma 5(b)]{BBH2001}, which deduces monotonicity of the map $t\mapsto J_t(\nu)/t$ from the sub-additivity of $t \mapsto J_t(\nu)$. However, such monotonicity cannot be derived from sub-additivity in general (though the reverse holds true). As a counter example, consider $f(t) = 1 + |\sin(1/t)|$ for $t>0$. This function is sub-additive but $f(t)/t$ fails to be non-increasing on any right-neighborhood of $0$.
\end{remark}

\subsection{Proof of Proposition~\ref{prop:restGamma-lsc}}\label{S:gammalsc}
The proof contains two parts.\\
(i) Let us begin with the first part of the statement. Let $K>0$. We first note that by Lemma 3.1 in~\cite{DV75}(I) one has that $\tilde J_\gep\leq \gep\tilde J$. Thus, it follows from Lemma~\ref{lem:Psieps-to-gamma} that $\tilde\Psi_\gep$ converges uniformly to $\tilde\Gamma \label{key}$ on $\cA(K) = \{\xi \in \tilde \cX\colon \tilde J(\xi) \le K\}$, as $\gep\to 0$. If $\tilde\Psi_\gep$ restricted to $\cA(K)$ is lower semi-continuous for all $\gep>0$ we deduce therefore that $\tilde\Gamma$ restricted to $\cA(K)$ is also lower semi-continuous. It remains to prove that $\tilde\Psi_\gep$ restricted to $\cA(K)$ is lower semi-continuous. This follows from the three following facts:
	\begin{itemize}
		\item $\Psi_\gep(\mu) = \phi_{\infty,\gep}(\mu \otimes \pi_{ \gep/d})$ for all $\mu\in \cM_{\leq 1}(\bbR^d)$: see~\cite[Proof of Lemma 6]{BBH2001} and use the fact that $(\mu\otimes \pi_{ \gep/d})(\bbR^d\times\bbR^d) = \mu(\bbR^d)$; 
		\item $\tilde\phi_{\infty,\gep}$ is lower semi-continuous for the $\bfD_2$-topology on $\tilde\cX^{(2)}$(Lemma~\ref{lem:lsc-phi});
		\item $\xi \in \tilde\cX \mapsto \xi\otimes \pi_{ \gep/d} \in \tilde\cX^{(2)}$ is continuous for the $\bfD$-topology on $\tilde\cX$. Here, when $\xi=\{\tilde{\alpha}_i\}_{i\in I}\in\tilde\cX$ we mean $\xi\otimes \pi_{ \gep/d}:=\{\tilde{\alpha_i\otimes \pi_{ \gep/d}}\}_{i\in I}$.
	\end{itemize}

	To prove the last point, pick $f\colon (\bbR^d\times \bbR^d)^n\to \bbR\in \cF_{2n}$ and consider $\xi=\{\tilde\alpha_i\}_{i\in I}\in \tilde\cX$. 
Define, for $x_1,\ldots, x_n \in \bbR^d$,
\beq
g(x_1,\ldots, x_n) := \int_{(\R^d)^n}f(x_1, y_1,\ldots, x_n,y_n)\prod_{j=1}^np_{ \gep/d}(x_j,y_j)\dd y_j,
\eeq	
so that
	\beq\label{eq:from-f-to-g}
	\int_{(\R^d\times\R^d)^n} f(x_1, y_1,\ldots, x_n,y_n) \prod_{j=1}^n(\alpha_i\otimes \pi_{ \gep/d})(\dd x_j, \dd y_j)
	=: \int_{(\R^d)^n}g(x_1,\ldots, x_n)\prod_{j=1}^n\alpha_i(\dd x_j)\,.
	\eeq
Let us prove that $g\in\cF_n$. It is immediate that $g$ is translation invariant and that it vanishes at infinity. Regarding continuity, let $(x_1,\ldots, x_n)\in (\R^d)^n$ such that $|x_j|\leq M$ for all $j\in\{1,\ldots,n\}$ and some $M>0$. Then, for all $j\in\{1,\ldots,n\}$,
	\beq
	p_{ \gep/d}(x_j,y_j)\le \frac{\textrm{(cst)}}{\gep^{d/2}}\exp\Big(-\tfrac{{ d}|y_j|^2}{2\gep}+ \tfrac{{ d}M|y_j|}{2\gep}\Big)\,,
	\eeq
	which is integrable. The continuity of $g$ now follows from the boundedness and continuity of $f$, and dominated convergence. From~\eqref{eq:from-f-to-g} and the fact that $g\in\cF_n$, we may now conclude that
	\beq
	\xi\mapsto \sum_{i\in I} \int_{(\R^d\times\R^d)^n}f(x_1,y_1,\ldots, x_n,y_n)\prod_{j=1}^n(\alpha_i\otimes \pi_{ \gep/d})(\dd x_j, \dd y_j)
	\eeq
	is continuous, which completes this part of the proof.\\

	(ii) Let us now prove the second part of the statement, namely that $\tilde\Gamma_\gd$ is lower semi-continuous.
	The proof of this part follows closely the proof of Lemma~\ref{lem:lsc-phi}, so we will not provide all the details. First of all we note that in the same way as in the proof of Lemma~\ref{lem:lsc-phi} we can write for any sub-probability measure $\mu\in\cM_{\leq 1}(\R^d)$,

	\beq\label{eq:Gammadeltaexpansion}
	\tilde\Gamma_{ d\delta}(\{\tilde\mu\}) = \gk -\sum_{n\geq 2}\frac{(-\gk)^n}{n!}\int_{(\R^d)^n}W_{n}^{\delta}(x_1,x_2,\ldots, x_n)\prod_{i=1}^{n}\mu(\dd x_i)\,,
	\eeq
with the translation invariant function:
	\beq
	W_{n}^{\delta}(x_1,x_2,\ldots, x_n):=\int_{\R^d}\dd z \prod_{i=1}^{n}p_\delta(z-x_i)\,.
	\eeq
	In view of the proof of Lemma~\ref{lem:lsc-phi} and of the boundedness of the heat kernel $p_\delta$ it is enough to show that $W_n^\delta$ is continuous and vanishes at infinity. Concerning the continuity, assume that for some $M>0$ we have that $|x_i|\leq M$ for all $i\in\{1,2,\ldots,n\}$. Then for all $i$ we can estimate
	\beq
	p_\delta(z-x_i)\le \frac{\textrm{(cst)}}{\delta^{d/2}}\exp\Big(-\tfrac{|z|^2}{2\delta}+\tfrac{M|z|}{2\delta}\Big)\,,
	\eeq
	which is integrable in $z$. Since $p_\delta$ is continuous, the continuity of $W_n^\delta$ follows by dominated convergence.
	It remains to show that $W_n^\delta$ vanishes at infinity, in the sense that
	\beq
	\lim_{\max_{i\neq j}|x_i-x_j|\to\infty} W_n^\delta(x_1,x_2,\ldots, x_n)=0\,.
	\eeq To that end, and without any loss of generality, we can assume that $|x_1-x_2|\to\infty$\,. Let $M=|x_1-x_2|$ and define
	\beq
	\begin{aligned}
		A_1&= \{z\in\R^d\,:|z-x_1|<\tfrac{M}{2}, |z-x_2|\geq \tfrac{M}{2}\}\,,\\
		A_2&=\{z\in\R^d\,:|z-x_1|\geq\tfrac{M}{2}, |z-x_2|< \tfrac{M}{2}\}\,,\\
		A_3&=\{z\in\R^d\,:|z-x_1|\geq\tfrac{M}{2}, |z-x_2|\geq \tfrac{M}{2}\}\,.
	\end{aligned}
	\eeq
	By the choice of $M$, it follows that $A_1\cup A_2\cup A_3=\R^d$\,. On $A_1$ we can estimate
	\beq
	p_\delta(z-x_2)\le \frac{\textrm{(cst)}}{\delta^{d/2}}e^{-\tfrac{M^2}{8\delta}}\,.
	\eeq
	Thus, by the boundedness of $p_\delta$ is follows that
	\beq
	\int_{A_1}\dd z\prod_{i=1}^n p_\delta(z-x_i)\le  \frac{\textrm{(cst)}}{\delta^{d/2}}e^{-\tfrac{M^2}{8\delta}}\int\dd z\, p_\delta(z-x_1)=\frac{\textrm{(cst)}}{\delta^{d/2}}e^{-\tfrac{M^2}{8\delta}}\,,
	\eeq
	which tends to zero as $M\to\infty$. The integrals over $A_2$ and $A_3$ can be dealt with in the same manner. This concludes the proof.
\subsection{Proof of Proposition~\ref{pr:LB-JepsK}}\label{S:LB-JepsK}
The proof relies on the following lemma. 
\begin{lemma}
	\label{pr:LB-Jepsn}
	Suppose $(\xi_n) \to \xi = \{\tilde\ga_i\}_{i\in I}\in \tilde\cX$. Then, for every sequence $(\gep_n)_{n\in\N}$ tending to zero, and for every $\gd>0$, there exists $k =k(\gd)$ such that for every collection of positive constants $(c_i)_{1\le i \le k}$ and every collection $(v_i)_{1\le i \le k}$ of smooth non-negative functions with compact support,
	\beq\label{eq:LB-Jepsn}
	\liminf_{n\to\infty}
	\frac{1}{\gep_n} \tilde J_{\gep_n}(\xi_n) \ge
	- \sum_{1\le i \le k} \int \frac{\frac12 \gD v_i(x)}{c_i + v_i(x)} \ga_i(\dd x) - \gd\,.
	\eeq
	Moreover, $k=k(\delta)$ can be chosen such that $\delta\mapsto k(\delta)$ is non-increasing with $\lim_{\delta\to 0}k(\delta)=|I|$, where the cardinality $|I|$ of the index set $I$ may be either finite or infinite.
\end{lemma}
For the proof of the above statement we recall Equation~\eqref{eq:varrepresentation}. We also recall
the following standard fact about the entropy rate function:
\beq
J(\ga) = \sup_{u,c} \int -\frac12\frac{\gD u}{c+u} \dd \ga, \qquad \ga\in\cM_{\le 1}(\bbR^d),
\eeq
where the supremum runs over $c\in(0,\infty)$ and non-negative, smooth, and compactly supported functions $u$, see~\cite[Proof of Lemma 4.7]{MV2016}.
By optimizing over the $c_i$'s and $v_i$'s in~\eqref{eq:LB-Jepsn} and letting $\gd\to 0$, we thus obtain as an immediate corollary of Lemma~\ref{pr:LB-Jepsn}:
\begin{corollary}
\label{cor:LB-Jepsn}
Suppose $(\xi_n) \to \xi = \{\tilde\ga_i\}_{i\in I}\in \tilde\cX$. Then, for every sequence $(\gep_n)_{n\in\N}$ tending to zero,
	\beq
	\liminf_{n\to\infty}
	\frac{1}{\gep_n} \tilde J_{ \gep_n/d}(\xi_n) \ge
	{ \frac1d}\tilde J(\xi).
	\eeq
\end{corollary}

\begin{proof}[Proof of Lemma~\ref{pr:LB-Jepsn}]
	The proof strategy is inspired by Lemma 4.2 in~\cite{MV2016}. Let us first consider the simpler case when, for all $n$, $\xi_n$ consists of a single orbit denoted by $\tilde \mu_n$. Then, as in~\cite{MV2016} we may write at least along some subsequence, which we will suppress from the notation
	\beq
	\mu_n = \sum_{1\le i\le k} \ga_n^{(i)} + \gb_n.
	\eeq
	with
	\beq
\ga_n^{(i)} * \gd_{a_n^{(i)}} \stackrel{\textrm{(weakly)}}{\longrightarrow} \ga_i, \quad
\qquad n\to\infty,
	\eeq
and
\beq
\liminf_{n\to\infty} \inf_{i\neq j} |a_n^{(i)}- a_n^{(j)}| = +\infty
\eeq
	for some sequences $(a_n^{(i)})$ in $\bbR^d$. Moreover, for all $V\in \cF_2$,
	\beq
	\ba
	&\int V(x,y) \ga_n^{(i)}(\dd x) \gb_n(\dd y) \stackrel{n\to\infty}{\longrightarrow} 0,\\
	&\limsup_{n\to\infty} \int V(x,y) \gb_n(\dd x) \gb_n(\dd y) \le \gd.
	\ea
	\eeq
	Moreover, $k$ depends monotonously on $\delta$ and tends to $|I|$ as $\delta$ tends to zero.
	For any function $u$ which is smooth and positive,
	\beq
	\frac 1\gep \log \Big(\frac{\pi_\gep u}{u}\Big)(x) \le
	\Big(\frac{\frac12\gD u}{u}\Big)(x) +
	\Big| \frac{\pi_\gep u -u}{\gep u}(x) -  \frac{\frac12\gD u}{u}(x)\Big|.
	\eeq
	By integrating this inequality w.r.t. $\mu_n$, we obtain
	\beq
	\ba
	\int \frac1\gep \log \Big(\frac{\pi_\gep u}{u}\Big)\mu_n(\dd x) \le 
	&\sum_{1\le i\le k} \int \Big(\frac{\frac12\gD u}{u}\Big)(x) \ga_n^{(i)}(\dd x)
	+ \int \Big(\frac{\frac12\gD u}{u}\Big)(x) \gb_n(\dd x)\\
	&\qquad+ \int \Big| \frac{\pi_\gep u -u}{\gep u}(x) -  \frac{\frac12\gD u}{u}(x)\Big| \mu_n(\dd x).
	\ea
	\eeq
	We now pick $u$ of the form
	\beq
	u(x) = \sum_{1\le i\le k} \Big[c_i +v_i(x+a_n^{(i)}) \varphi\Big(\frac{x+ a_n^{(i)}}{R}\Big)\Big],
	\eeq
	where
	\begin{itemize}
	\item $R>0$,
		\item $(c_i)_{1\le i\le k}$ are positive numbers,
		\item $(v_i)_{1\le i\le k}$ are non-negative, smooth and compactly supported functions,
		\item $\varphi\colon \bbR^d\to [0,1]$ is smooth and satisfies $\varphi(x) = 0$ if $\|x\|\ge 2$ and $\varphi(x) = 1$ if $\|x\|\le 1$.
	\end{itemize}
	We abbreviate $v_{i,R}(\cdot) = v_i(\cdot) \varphi(\frac{\cdot}{R})$, so that
	\beq
	\ba
	u_n(x) = \sum_{1\le i\le k} (c_i + v_{i,R}(x+ a_n^{(i)})),\\
	\gD u_n(x) = \sum_{1\le i\le k} \gD v_{i,R}(x+ a_n^{(i)}).
	\ea
	\eeq
	Then,
	\beq
	\ba
	&\limsup_{n\to\infty} \int \Big(\frac{\frac12\gD u_n}{u_n}\Big)(x) \ga_n^{(i)}(\dd x)
	= \int \Big(\frac{\frac12\gD v_{i,R}}{c_i+v_{i,R}}\Big)(x) \ga_i(\dd x),\\
	&\limsup_{n\to\infty} \int \Big(\frac{\frac12\gD u_n}{u_n}\Big)(x) \gb_n(\dd x)
	\le \gd.
	\ea
	\eeq
	Moreover,
	\beq
	\Big| \frac{\pi_\gep u_n -u_n}{\gep u_n}(x) -  \frac{\frac12\gD u_n}{u_n}(x)\Big| \le 
	 \sum_{1\le i\le k} \frac 1{c_i}
	\Big\|\frac{\pi_\gep v_{i,R} -v_{i,R}}{\gep} - \frac12\gD v_{i,R}\Big\|_\infty,
	\eeq
	which converges to $0$ as $\gep\to 0$. Note that the convergence is uniform because the $v_{i,R}$'s are in the domain of the generator $\frac12\gD$ (see e.g.~\cite[Chapter 6.2]{LeGall-en}). Finally, we obtain
	\beq
	\limsup_{n\to\infty} \frac{1}{\gep_n} \int
	\log \frac{\pi_{\gep_n} u_n}{u_n}(x)\mu_n(\dd x)
	\le \int \sum_{1\le i\le k} \Big(\frac{\frac12\gD v_{i,R}}{c_i+v_{i,R}}\Big)(x) \ga_i(\dd x)
	+ \gd.
	\eeq
	This completes the proof in the single-orbit case.
	
	\par If for all $n\ge 1$, $\xi_n$ consists of several orbits denoted by $\xi_n^{(j)}$, we can choose a subsequence such that, for all $j$, $\xi_n^{(j)}$ converges to $\{\tilde \ga_i^{(j)}, i\in I\}$. From the single-orbit case, we get, for all $m\ge 1$,
\beq
\liminf_{n\to\infty} \frac{1}{\gep_n} \tilde J_{\gep_n}(\xi_n) \ge \liminf_{n\to\infty} \sum_{j=1}^m  \frac{1}{\gep_n} \tilde J_{\gep_n}(\xi_n^{(j)}) \ge \sum_{j=1}^m \sum_{i\in I} J(\ga_i^{(j)}).
\eeq
Letting $m\to\infty$, we get the result.
\end{proof}

\begin{proof}[Proof of Proposition~\ref{pr:LB-JepsK}]
Let $(\gep_n)$ be a sequence of positive real numbers converging to zero. From~Corollary~\ref{cor:LB-Jepsn} we obtain that
	\beq
	\liminf_{n\to\infty}
	\frac{1}{\gep_n} \tilde J_{ \gep_n/d}(\xi_n) \ge { \frac1d}\tilde J(\xi)
	\eeq
	as soon as $\xi_n$ converges to $\xi$ in $\tilde\cX$.
By compactness of $\cK$ and lower semi-continuity of the function $\tilde J_\gep$, there exists $\xi_n \in \cK$ such that
	\beq
	\inf_{\tilde \mu \in \cK}
	\frac 1{\gep_n} \tilde J_{ \gep_n/d}(\tilde \mu) = \frac 1{\gep_n} \tilde J_{ \gep_n/d}(\xi_n).
	\eeq
	and any subsequence of $(\xi_n)$ has a limit (along a further subsequence) in $\cK$, which we will denote by $\xi$. From what precedes, 
	\beq
	\liminf_{n\to\infty} \frac 1{\gep_n} \tilde J_{ \gep_n/d}(\xi_n) \ge { \frac1d}\tilde J(\xi) \ge { \frac1d}\inf_\cK \tilde J.
	\eeq
\end{proof}
\appendix

\section{Technical estimates}
\subsection{Large Deviation estimate for the binomial distribution}
\begin{lemma}
\label{lem:LDest}
For all $M\in\bbN$ and $C>1\vee \gep_0>0$, 
\beq
\bP(\bin(M,\gep_0/C)\ge \gep_0M) \le \exp(-M\gep_0 \Xi(C)),
\eeq
with $\Xi(C):=\log C+\tfrac1C-1$.
\end{lemma}
\begin{proof}[Proof of Lemma~\ref{lem:LDest}]
By Chernov's bound we may write for all $\gl>0$,
\beq
\ba
\bP(\bin(M,\gep_0/C)\ge \gep_0M) &\le 
\exp\Big(-M\Big[\gl\gep_0 - \log(1+ \tfrac{\gep_0}{C}(e^\gl-1))\Big]\Big)\\
&\le 
\exp\Big(-\gep_0M\Big[\gl - \tfrac{1}{C}(e^\gl-1)\Big]\Big),
\ea
\eeq
and we conclude by picking $\gl = \log C$.
\end{proof}

\subsection{Concentration inequality}
\label{sec:app-con-ine}
\begin{proof}[Proof of Proposition~\ref{pr:conc.ineq}]
The proof follows the same lines of argument as in~\cite[Proposition 2.2.2]{Phetpradap}, where the same result was proven for simple random walk on the torus. Thus, instead of giving a complete proof, we only point out the differences. In Equation (2.2.12) in~\cite{Phetpradap} the estimate 
\begin{equation}
\frac{1}{n} \hat{\cR}_{n,\gep}^K \leq N^d\,,
\end{equation}
should be replaced by
\begin{equation}
\frac{1}{n}\hat{\cR}_{n,\gep}^K\leq 1\,,
\end{equation}
which follows immediately from the definition of $\hat{\cR}_{n,\gep}^K$. Moreover, at the beginning of Step (4) the crude estimate

\begin{equation}
\frac{1}{n} \Big|\cR_{n,\gep}^K-m_{n,\gep}^K\Big| \leq 1
\end{equation}
should be used which implies exactly as in~\cite{Phetpradap} that actually
\begin{equation}
\frac{1}{n} \Big|\bE_{n,\gep}\cR_{n,\gep}^K-m_{n,\gep}^K\Big| \leq \frac{\delta}{3} + \bP_{n,\gep}\Big(\frac{1}{n} \Big|\cR_{n,\gep}^K-m_{n,\gep}^K\Big|\geq \frac{\delta}{3}\Big)\,.
\end{equation}
From that point on the proof follows exactly~\cite{Phetpradap}.
\end{proof}
\section{Proof of Proposition~\ref{pr:app-cond-range}}
\label{sec:app-proof-cond}
For the proof of Proposition~\ref{pr:app-cond-range} we will need to make use of the following local central limit theorem.
\begin{proposition}[Local limit theorem, see Theorem 1.2.1 in~\cite{Lawler2013}]
	\label{pr:LLT}
	Recall that $p_s(x) = (2\pi s)^{-\frac d2}\exp(-\frac{|x|^2}{2s})$ for $s>0$ and $x\in\bbR ^d$. Define $E(n,x) = \bP(S_n=x) - 2p_{n/d}(x)$ if $n$ and $x$ have the same parity, and zero otherwise. Then
	\beq
	|E(n,x)|\le O(n^{-\frac{d+2}{2}}) \wedge |x|^{-2}O(n^{-\frac{d}{2}}),
	\eeq
	where $u_n = O(v_n)$ means $u_n\le C v_n$ for some finite positive constant $C$ and all $n$.
\end{proposition}
We also need to introduce truncated versions of some of the functions defined in Section~\ref{sec:cond-range}, namely:
\beq
\frak{q}_{\gep, n, \gL}^\rho(x,y):= -M \log[1-\bar q_{\ell, \gL}^\rho(x, y)],
\eeq 
where for $\rho>0$,
\beq
\ba
\bar q_{\ell, \gL}^\rho(x, y) =\begin{cases}
	\bar q_{\ell, \gL}(x, y), &\text{if } d(\gL, \{\lfloor xn^{1/d}\rfloor, \lfloor yn^{1/d}\rfloor\})> \rho n^{1/d} ,\\
	0, &\text{otherwise,}
\end{cases}
\ea
\eeq
and with the convention $\bar q_{\ell,\Lambda}^0=\bar q_{\ell,\Lambda}$. When $\gL=\gL(n,z)$ is the singleton $\{\lfloor zn^{1/d}\rfloor\}$ ($z\in\bbR^d$) we define
\beq
\bar q_{\ell}^\rho(x,y;z) := \bar q_{\ell,\gL(n,z)}^\rho(x,y),
\qquad x,y\in \bbR^d.
\eeq
We also define
\beq\label{def:varphi.rho}
\varphi^\rho_\gep(x,y) =
\begin{cases}
\varphi_\gep(x,y) & {\rm if}\ x,y\notin \bar\cB(0,\rho)\\
0 & {\rm otherwise,}
\end{cases}
\eeq
where $\bar\cB(0,\rho)$ is the (closed) Euclidean ball in $\bbR^d$ with radius $\rho$ and centered at the origin.
Finally, we define for all $\mu\in \cM_1(\bbR^d\times\bbR^d)$, $n\in\bbN$ and $\rho\ge 0$:
\beq
\label{eq:def-phi-n-rho}
\phi_{n,\gep,\rho}(\mu)
=\int  \dd z
\Big(1 - \exp\Big\{-\frac1\gep \int \mu(\dd x, \dd y) n^{1-\frac2d} \bar q^\rho_{\ell}(x,y;z)\Big\}\Big)
\eeq
and
\beq
\label{eq:def-phi-infty-rho}
\phi_{\infty,\gep,\rho}(\mu)
=\int  \dd z
\Big(1 - \exp\Big\{-\frac{\gk_d}{\gep} \int \mu(\dd x, \dd y) \varphi_\gep^\rho(x-z,y-z)\Big\}\Big),
\eeq
with the conventions that $\phi_{\infty,\gep,0}= \phi_{\infty,\gep}$.\\
\par We now come to the proof of Proposition~\ref{pr:app-cond-range}.
We proceed in several steps.\\
{\it Step 1. Removing the logarithm.} By~\eqref{eq:exp-cond-range}, we get
\beq
\tfrac1n \bE(R_n|\hat S) = \boldsymbol\Phi_{n,\gep,0}(L_{M,\gep}^{(2)}),
\eeq
where 
\beq
\boldsymbol\Phi_{n,\gep,0}(\mu) := \int_{\bbR^d} \dd z
\Big[
1- \exp\Big(- \int \frak{q}_{\gep,n}(x,y;z) \mu(\dd x, \dd y)\Big)
\Big].
\eeq
Recall the definition in~\eqref{eq:def-phi-n-rho}.
Since $\log(1+x)\le x$ for all $x>-1$,
\beq
\boldsymbol\Phi_{n,\gep,0}(L_{M,\gep}^{(2)}) \ge \phi_{n,\gep,0}(L_{M,\gep}^{(2)}).
\eeq
Finally, we get that 
\beq
\label{eq:LBCR1}
\tfrac1n \bE(R_n|\hat S) \ge \phi_{n,\gep,0}(L_{M,\gep}^{(2)})\,.
\eeq
In the following steps we approximate $\phi_{n,\gep,0}(L_{M,\gep}^{(2)})$ by
$\phi_{\infty,\gep,0}(L_{M,\gep}^{(2)})$.
\\

{\it Step 2. Truncation of the range and decomposition of the error term.} 
For any $K>0$ (whose precise value will be determined later), we have the trivial bound
\beq
\phi_{n,\gep,0}(\mu) \ge \phi_{n,\gep,Kn^{-1/d}}(\mu).
\eeq
Therefore, we may write on the event $L_{M,\gep}^{(2)}\in\cM(A,\gep_0)$:
\beq
\label{eq:LBCR2}
\phi_{n,\gep,0}(L_{M,\gep}^{(2)}) \ge 
\phi_{\infty,\gep,0}(L_{M,\gep}^{(2)}) - (\cE_1 + \cE_2)(n,K),
\eeq
where
\beq
\ba
&\cE_1(n,K) := \sup_{\mu\in\cM(A,\gep_0)} [\phi_{\infty,\gep,0}(\mu) - \phi_{\infty,\gep,Kn^{-1/d}}(\mu)],\\
&\cE_2(n,K) := \sup_{\mu\in\cM(A,\gep_0)} [\phi_{\infty,\gep,Kn^{-1/d}}(\mu) - \phi_{n,\gep,Kn^{-1/d}}(\mu)].
\ea
\eeq
Note that there is no need for absolute values inside the supremum since we only aim at a lower bound. Recall that $\gep,\gep_0$ and $A$ are fixed. We deal with the error terms $\cE_1(n,K)$ and $\cE_2(n,K)$ in Steps 3 and 4, respectively. 
\\

{\it Step 3. Control on $\cE_1(n,K)$.} We will prove that for any $K>0$,
\beq
\label{eq:LBCR3}
\limsup_{n\to\infty} \cE_1(n,K) \le 2\gk\gep_0.
\eeq
Therefore, throughout this proof step, the value of $K$ is fixed (the precise value will be decided in Step 4 below).
For all $\mu\in\cM_1(\bbR^d\times \bbR^d)$, we have
\beq
\phi_{\infty,\gep,0}(\mu) - \phi_{\infty,\gep,Kn^{-1/d}}(\mu)
\le \int\mu(\dd x, \dd y) \frac{\gk}{\gep}
\int_{\cB(x,Kn^{-\frac1d})\cup\cB(y,Kn^{-\frac1d})}
\varphi_\gep(x-z,y-z)\dd z.  
\eeq
Recalling Lemma~\ref{lem:phi.q.prob}, we obtain that for all $\mu\in\cM(A,\gep_0)$,
\beq
\phi_{\infty,\gep,0}(\mu) - \phi_{\infty,\gep,Kn^{-1/d}}(\mu)
\le \gk\gep_0 + \cE_{1,1}(n,K,\mu) + \cE_{1,2}(n,K,\mu),
\eeq
where
\beq
\cE_{1,1}(n,K,\mu) :=
\int_{|x-y|\le A}\mu(\dd x, \dd y) \frac{\gk}{\gep}
\int_{\cB(x,Kn^{-\frac1d})}
\varphi_\gep(x-z,y-z)\dd z
\eeq
and $\cE_{1,2}(n,K,\mu)$ is defined in the same way, with $\cB(x,Kn^{-\frac1d})$ replaced by $\cB(y,Kn^{-\frac1d})$. Since $x$ and $y$ play symmetric roles, it is enough to deal with $\cE_{1,1}(n,K,\mu)$. Recalling~\eqref{def:varphi}, we see that
\beq
\ba
\cE_{1,1}(n,K,\mu) &=
\int_{|x-y|\le A}\mu(\dd x, \dd y) \frac{\gk}{\gep}
\int_{\cB(x,Kn^{-\frac1d})} \dd z
\int_0^\gep \dd s \frac{p_{s/d}(x-z)p_{(\gep-s)/d}(z-y)}{p_{\gep/d}(x-y)}\\
&=
\gk\int_{|x-y|\le A}\mu(\dd x, \dd y)
\int_{\cB(x,Kn^{-\frac1d})} \dd z
\int_0^1 \dd s \frac{p_{\gep s/d}(x-z)p_{\gep(1-s)/d}(z-y)}{p_{\gep/d}(x-y)}.
\ea
\eeq
By using~\eqref{eq:phi.q.prob}, we get that for all $\mu\in\cM(A,\gep_0)$,
\beq
\cE_{1,1}(n,K,\mu) \le \gk(\gep_0 + \cE_{1,1}(n,K,\mu,\gep_0)),
\eeq
where
\beq
\cE_{1,1}(n,K,\mu,\gep_0) := \int_{|x-y|\le A}\mu(\dd x, \dd y)
\int_{\cB(x,Kn^{-\frac1d})} \dd z
\int_0^{1-\gep_0} \dd s \frac{p_{\gep s/d}(x-z)p_{\gep(1-s)/d}(z-y)}{p_{\gep/d}(x-y)}.
\eeq
Using the expression of the Brownian kernel, there exists $C=C(A,\gep_0,\gep)$ such that
\beq
\suptwo{s\le 1-\gep_0,\ z\in\bbR^d}{|x-y|\le A} \frac{p_{\gep(1-s)/d}(z-y)}{p_{\gep/d}(x-y)} \le C.
\eeq
Therefore,
\beq
\ba
\cE_{1,1}(n,K,\mu,\gep_0) &\le C \int_{|x-y|\le A}\mu(\dd x, \dd y) \int_0^{1-\gep_0}  \bP(|B_{\gep s/d}| \le Kn^{-\frac1d})\dd s\\
&\le C \int_0^1 \bP(|B_{\gep s/d}| \le Kn^{-\frac1d})\dd s,
\ea
\eeq
which converges to $0$ as $n\to \infty$, by dominated convergence. This completes the proof of Step 3.\\

{\it Step 4. Control on $\cE_2(n,K)$.} We will prove that there exists $K>0$ (large enough) such that
\beq
\label{eq:LBCR4}
\limsup_{n\to \infty} \cE_2(n,K) \le 4\gep_0.
\eeq
For convenience, we define
\beq
\cE_2(n,K,\mu) := \phi_{\infty,\gep,Kn^{-1/d}}(\mu) - \phi_{n,\gep,Kn^{-1/d}}(\mu)
\eeq
so that
\beq
\cE_2(n,K) = \sup_{\mu\in\cM(A,\gep_0)} \cE_2(n,K,\mu).
\eeq
This step of the proof is more involved, so we divide it into several smaller steps.\\

{\it Step 4a. Approximation of the escape probability.} Let us further estimate $\cE_2(n,K,\mu)$ as follows:
\beq
\cE_2(n,K,\mu) \le \cE_{2,1}(n,K,\mu) + \cE_{2,2}(n,K,\mu),
\eeq
where
\beq
\ba
&\cE_{2,1}(n,K,\mu) := \phi_{\infty,\gep,Kn^{-1/d}}(\mu) - \phi^{(K)}_{\infty,\gep,Kn^{-1/d}}(\mu),\\
&\cE_{2,2}(n,K,\mu) := \phi^{(K)}_{\infty,\gep,Kn^{-1/d}}(\mu) - \phi_{n,\gep,Kn^{-1/d}}(\mu),
\ea
\eeq
and
\beq
\label{eq:def-phi-infty-rho-K}
\phi^{(K)}_{\infty,\gep,\rho}(\mu)
=\int  \dd z
\Big(1 - \exp\Big\{-\frac{\bP(H_0> H_{\partial B_K})}{\gep} \int \mu(\dd x, \dd y) \varphi_\gep^\rho(x-z,y-z)\Big\}\Big)
\eeq
(compare with~\eqref{eq:def-phi-infty-rho}). We used the following notations for hitting times in the formula above:
\beq
H_{\gL} := \inf\{n\ge 1 \colon S_n\in \gL\}, \qquad \gL\subset \bbZ^d,
\qquad H_{a} = H_{\{a\}}, \quad a\in\bbZ.
\eeq
By Lemma~\ref{lem:phi.q.prob}, we obtain
\beq
\suptwo{\mu\in\cM(A,\gep_0)}{n\ge 1} \cE_{2,1}(n,K,\mu) \le |\gk - \bP(H_0> H_{\partial B_K})| = |\bP(H_0 = \infty) - \bP(H_0> H_{\partial B_K})|\stackrel{K\to\infty}{\longrightarrow} 0.
\eeq
Therefore, there exists $K_1(\gep_0)$ such that for all $K\ge K_1(\gep_0)$,
\beq
\label{eq:cE21}
\suptwo{\mu\in\cM(A,\gep_0)}{n\ge 1} \cE_{2,1}(n,K,\mu) \le \gep_0.
\eeq
It now remains to control $\cE_{2,2}(n,K,\mu)$ uniformly in 
$\mu\in\cM(A,\gep_0)$, as $n\to\infty$, which we will do in the following steps.\\

{\it Step 4b. Truncation of the time interval}. We now proceed to a truncation of the time interval $[0,\gep]$ appearing in the definition of $\varphi_\gep$. This truncation allows us to obtain uniform bounds later in the proof. We split $\cE_{2,2}(n,K,\mu)$ as follows:
\beq
\cE_{2,2}(n,K,\mu) \le \cE_{2,3}(n,K,\mu) + \cE_{2,4}(n,K,\mu),
\eeq
where
\beq
\ba
&\cE_{2,3}(n,K,\mu) := \phi^{(K)}_{\infty,\gep, Kn^{-1/d}}(\mu) -\phi^{(K,\gep_0)}_{\infty,\gep, Kn^{-1/d}}(\mu),\\
&\cE_{2,4}(n,K,\mu) := \phi^{(K,\gep_0)}_{\infty,\gep, Kn^{-1/d}}(\mu) - \phi^{(\gep_0)}_{n,\gep, Kn^{-1/d}}(\mu),\\
\ea
\eeq
with
\beq
\ba
\label{eq:def-phi-infty-rho-K-gep0}
&\phi^{(\gep_0)}_{n,\gep, \rho}(\mu) := \int \dd z\Big(1 - \exp\Big\{-\frac1\gep \int \mu(\dd x, \dd y) n^{1-\frac2d} \bar q^\rho_{\ell,\gep_0}(x,y;z)\Big\}\Big)\\
&\phi^{(K,\gep_0)}_{\infty,\gep,\rho}(\mu)
:=\int  \dd z
\Big(1 - \exp\Big\{-\frac{\bP(H_0> H_{\partial B_K})}{\gep} \int \mu(\dd x, \dd y) \varphi_{\gep,\gep_0}^\rho(x-z,y-z)\Big\}\Big)
\ea
\eeq
(compare with~\eqref{eq:def-phi-infty-rho-K}) and
\beq\label{def:varphi-gep0}
\ba
\varphi_{\gep,\gep_0}(x,y) &:= \int_{\gep_0\gep}^{(1-\gep_0)\gep} \dd s \frac{p_{s/d}(-x)p_{(\gep-s)/d}(y)}{p_{\gep/d}(y-x)},\\
\bar q_{\ell,\gep_0}(x,y;z) &:= \bP_{\lfloor xn^{\frac1d}\rfloor}(H_{\lfloor zn^{\frac1d}\rfloor} \in [\gep_0\ell,(1-\gep_0)\ell] | S_\ell = \lfloor yn^{\frac1d}\rfloor).
\ea
\eeq
The super-index $\rho$ in~\eqref{eq:def-phi-infty-rho-K-gep0} means that the corresponding function is zero as soon as one of its arguments is in the ball of radius $\rho$.
By using~\eqref{eq:phi.q.prob}, we readily get
\beq
\suptwo{\mu\in\cM_1(\bbR^d\times\bbR^d)}{n,K\ge 1} \cE_{2,3}(n,K,\mu) \le 2\gep_0.
\eeq
It remains to deal with the term $\cE_{2,4}(n,K,\mu)$. Note that we may safely restrict the integrals over $x$ and $y$ to $|x-y|\le A$ in~\eqref{def:varphi-gep0} (up to some error term not larger than $\gep_0$) since $\mu\in\cM(A,\gep_0)$ (same argument as in Step 3).\\

{\it Step 4c. Decomposition of $\phi^{(\gep_0)}_{n,\gep, Kn^{-1/d}}(\mu)$.}  To lighten notations, we introduce
\beq
x_n := \lfloor xn^{\frac1d} \rfloor, \qquad
y_n := \lfloor yn^{\frac1d} \rfloor, \qquad
z_n := \lfloor zn^{\frac1d} \rfloor,
\eeq
and
\beq
\label{eq:def-uv}
u:= x_n - z_n, \qquad v := y_n- z_n.
\eeq
We may thus write
\beq
\bar q_{\ell,\gep_0}(x,y;z) = \frac{Q_{\ell}(u,v)}{\bP_u(S_\ell = v)},
\qquad Q_{\ell}(u,v):= \bP_u(H_0\in[\gep_0,1-\gep_0]\ell, S_{\ell}= v).
\eeq
We now proceed with several approximations of $Q_{\ell}(u,v)$. First, by decomposing on the value of $H_0$ and reversing time on the interval $[0,H_0]$, we obtain
\beq
Q_{\ell}(u,v) = \sum_{i \in [\gep_0,1-\gep_0]\ell} \bP_0(S_i = u, H_0>i) \bP_0(S_{\ell - i}= v).
\eeq
Recall that $|u|\ge K$. By decomposing on the value of $H_{\partial B_K}$ and the position of the walk at that time, we obtain
\beq
Q_{\ell}(u,v) = \sum_{i \in [\gep_0,1-\gep_0]\ell} \sumtwo{1\le j \le i}{w\in \partial B_K} \bP_0(S_j=w, H_{\partial B_K}=j < H_0)\bP_w(S_{i-j} = u, H_0>i-j) \bP_0(S_{\ell - i}= v).
\eeq
For any $a\in\bbN$ with $a\leq \gep_0\ell$ (to be determined later in the proof), we have $Q_{\ell}(u,v) \ge Q_{\ell,a}(u,v)$, where
\beq
Q_{\ell,a}(u,v) := \sum_{i \in [\gep_0,1-\gep_0]\ell} \sumtwo{1\le j \le a}{w\in \partial B_K} \bP_0(S_j=w, H_{\partial B_K}=j < H_0)\bP_w(S_{i-j} = u, H_0>i-j) \bP_0(S_{\ell - i}= v).
\eeq
We decompose the latter as
\beq
Q_{\ell,a}(u,v) = Q_{\ell,a}^{(1)}(u,v) - {\frak E}_{\ell,a}^{(1)}(u,v),
\eeq
where 
\beq
Q_{\ell,a}^{(1)}(u,v) := \sum_{i \in [\gep_0,1-\gep_0]\ell} \sumtwo{1\le j \le a}{w\in \partial B_K} \bP_0(S_j=w, H_{\partial B_K}=j < H_0)\bP_w(S_{i-j} = u) \bP_0(S_{\ell - i}= v),
\eeq
and
\beq
{\frak E}_{\ell,a}^{(1)}(u,v) := \sum_{i \in [\gep_0,1-\gep_0]\ell} \sumtwo{1\le j \le a}{w\in \partial B_K} \bP_0(S_j=w, H_{\partial B_K}=j < H_0)\bP_w(S_{i-j} = u, H_0 \le i-j) \bP_0(S_{\ell - i}= v).
\eeq
Note that so far we have ignored the fact that $P_0(S_m=q)$ is non-zero only if $m$ and $q$ have the same parity. We write $m\lra q$ in that case.

We now bound $Q_{\ell,a}^{(1)}(u,v)$ from below as follows:
\beq
Q_{\ell,a}^{(1)}(u,v) \ge Q_{\ell,a}^{(2)}(u,v) - {\frak E}_{\ell,a}^{(2)}(u,v)
\eeq
where
\beq
Q_{\ell,a}^{(2)}(u,v) := \bP_0(H_{\partial B_K}< H_0)\sumtwo{i \in [\gep_0,1-\gep_0]\ell}{i-\ell\lra v}
 \infthree{w\in \partial B_K}{1\le j\le a}{i-j\lra u-w}
\bP_w(S_{i-j} = u) \bP_0(S_{\ell - i}= v)
\eeq
and
\beq
{\frak E}_{\ell,a}^{(2)}(u,v) := \bP_0(H_{\partial B_K}\ge a)
\sumtwo{i \in [\gep_0,1-\gep_0]\ell}{i-\ell\lra v}
 \infthree{w\in \partial B_K}{1\le j\le a}{i-j\lra u-w}
\bP_w(S_{i-j} = u) \bP_0(S_{\ell - i}= v).
\eeq
The term $Q_{\ell,a}^{(2)}(u,v)$ is the term that will give the main contribution to $\cE_{2,4}(n,K,\mu)$. We will first control the terms ${\frak E}_{\ell,a}^{(1)}(u,v)$ and ${\frak E}_{\ell,a}^{(2)}(u,v)$ and then come back to $Q_{\ell,a}^{(2)}(u,v)$.\\

{\it Step 4d. Control of ${\frak E}_{\ell,a}^{(1)}(u,v)$.} Let us now deal with the error term ${\frak E}_{\ell,a}^{(1)}(u,v)$. We use the following lemma, the proof of which is postponed to the end of the section
\begin{lemma}
\label{lem:return-origin}
We have
\beq
\suptwo{|u|\wedge|w|\ge K}{k\ge 1} k^{\frac d2}\bP_w(S_k = u, H_0 \le k) \le {\rm (cst)} K^{2-d}.
\eeq
\end{lemma}
By Lemma~\ref{lem:return-origin}, the contribution from ${\frak E}_{\ell,a}^{(1)}(u,v)$ is at most:
\beq
\label{eq:step5d}
\ba
&\int_{|x-y|\le A}  \mu(\dd x, \dd y) \int \dd z \frac 1\gep n^{1-2/d} \frac{{\frak E}_{\ell,a}^{(1)}(u,v)}{\bP_u(S_{\ell}=v)}\\
&\le {\rm (cst)} K^{2-d} \int_{|x-y|\le A} \mu(\dd x, \dd y) \frac{1}{\gep n^{\frac2d}} \int n \dd z  \sumtwo{i \in [\gep_0,1-\gep_0]\ell}{i-\ell\lra v} (i-a)^{-\frac d2}\frac{\bP_0(S_{\ell-i} = v)}{\bP_u(S_{\ell}=v)}.
\ea
\eeq
By using the local limit theorem (Proposition~\ref{pr:LLT}) on $\bP_u(S_{\ell}=v)$ and noticing that 
\beq
\int n \dd z \bP_0(S_{\ell-i} = v) = \sum_{z\in\bbZ^d} \bP_0(S_{\ell-i} = y_n - z) = 1,
\eeq
we obtain
\beq
\eqref{eq:step5d} \le C(A,\gep_0,\gep)\Big(\frac{\ell}{\gep_0\ell - a}\Big)^{d/2} K^{2-d} \le C(A,\gep_0,\gep)\Big(\frac2 {\gep_0}\Big)^{d/2} K^{2-d},
\eeq
provided $n$ is larger than some $n_1(\gep_0,\gep, a)$. We may now choose $K\ge K_2(\gep_0,\gep,A)$ such that
\beq
\eqref{eq:step5d} \le \gep_0.
\eeq
In the following we set $K\ge K_2(\gep_0,\gep,A) \vee K_1(\gep_0)$ (recall~\eqref{eq:cE21}) once and for all.\\

{\it Step 4e. Control of ${\frak E}_{\ell,a}^{(2)}(u,v)$.} Let $w\in\partial B_K$ and $1\le j \le a$. Similarly to Step 4d, the total contribution from ${\frak E}_{\ell,a}^{(2)}(u,v)$ is at most
\beq
\label{eq:step5e}
\bP_0(H_{\partial B_K}\ge a) \int_{|x-y|\le A}\mu(\dd x,\dd y) \int \dd z \frac{1}{\gep} n^{1-\frac2d} \sumtwo{i \in [\gep_0,1-\gep_0]\ell}{i-\ell\lra v} \bP_w(S_{i-j}=u)\frac{\bP_0(S_{\ell-i}=v)}{\bP_u(S_\ell = v)}.
\eeq
By the local limit theorem (Proposition~\ref{pr:LLT}), there exists $C(A,\gep_0,\gep)$ such that
\beq
\frac{\bP_0(S_{\ell-i}=v)}{\bP_u(S_{\ell}=v)} \le C(A,\gep_0,\gep),
\eeq
uniformly in $i\le(1-\gep_0)\ell$, $|x-y|\le A$ and $z\in\bbR^d$. Therefore, the quantity in~\eqref{eq:step5e} is bounded from above by
\beq
C(A,\gep_0,\gep) \bP_0(H_{\partial B_K}\ge a) \int \mu(\dd x, \dd y)  \frac{1}{\gep n^{\frac2d}} \sumtwo{i \in [\gep_0,1-\gep_0]\ell}{i-\ell\lra v} \int n \dd z \bP_w(S_{i-j}=u).
\eeq
Since
\beq
\int n \dd z \bP_w(S_{i-j}=u) = \sum_{z\in\bbZ^d}\bP_w(S_{i-j}=x_n -z) = 1,
\eeq
we get that
\beq
\text{\eqref{eq:step5e}} \le C(A,\gep_0,\gep) \bP_0(H_{\partial B_K}\ge a).
\eeq
Since $A,\gep_0,\gep$ and $K$ are fixed, we may now pick $a$ large enough such that
\beq
\text{\eqref{eq:step5e}} \le \gep_0.
\eeq

{\it Step 4f. Control of $\cE_{2,4}(n,K,\mu)$.} We define
\beq
\tilde Q_{\ell,a}^{(2)}(u,v) := \frac{Q_{\ell,a}^{(2)}(u,v)}{\bP_0(H_{\partial B_K}< H_0)}.
\eeq
By the local limit theorem (Proposition~\ref{pr:LLT}), we get for all $i\lra \ell-v$,
\beq
\sup_v |P_0(S_{\ell-i}= v) - 2p_{(\ell-i)/d}(v)| \le \frac{{\rm (cst)}}{(\ell-i)^{\frac d2 + 1}} \le \frac{c(\gep_0)}{\ell^{\frac d2 + 1}}.
\eeq
We may thus write 
\beq
\tilde Q_{\ell,a}^{(2)}(u,v) = \tilde Q_{\ell,a}^{(3)}(u,v) - {\frak E}_{\ell,a}^{(3)}(u,v),
\eeq
with
\beq
\tilde Q_{\ell,a}^{(3)}(u,v) := 2\sum_{\substack{i \in [\gep_0,1-\gep_0]\ell\,\\ \ell -i \leftrightarrow v}}
\infthree{w\in \partial B_K}{1\le j\le a}{i-j\lra u-w}
\bP_w(S_{i-j} = u) p_{(\ell-i)/d}(v).
\eeq
and
\beq
{\frak E}_{\ell,a}^{(3)}(u,v) \le c(\gep_0) \frac 1\ell \sum_{\substack{i \in [\gep_0,1-\gep_0]\ell\,\\ \ell -i \leftrightarrow v}}
\infthree{w\in \partial B_K}{1\le j\le a}{i-j\lra u-w}
\bP_w(S_{i-j} = u) \ell^{-d/2}.
\eeq
The contribution of this error term is dealt with by using the same line of arguments as in Steps 4d and 4e, uniformly in $\mu\in\cM(A,\gep_0)$. Again, by the local limit theorem (Proposition~\ref{pr:LLT}), we may write
\beq
\tilde Q_{\ell,a}^{(3)}(u,v) = \tilde Q_{\ell,a}^{(4)}(u,v) - {\frak E}_{\ell,a}^{(4)}(u,v)
\eeq
with
\beq
\tilde Q_{\ell,a}^{(4)}(u,v) := 4\sum_{\substack{i \in [\gep_0,1-\gep_0]\ell\,\\ \ell -i \leftrightarrow v}}
\infthree{w\in \partial B_K}{1\le j\le a}{i-j\lra u-w}
p_{(i-j)/d}(u-w) p_{(\ell-i)/d}(v).
\eeq
and
\beq
{\frak E}_{\ell,a}^{(4)}(u,v) \le c(\gep_0) \frac 1\ell \sum_{\substack{i \in [\gep_0,1-\gep_0]\ell\,\\ \ell -i \leftrightarrow v}} p_{(\ell - i)/d}(v) \ell^{-d/2}.
\eeq
To control the contribution from this error term, we use that
\beq
\int_{\bbR^d} n\dd z\ p_{k}(v) = \sum_{z\in\bbZ^d} p_k(z) \stackrel{k\to\infty}{\longrightarrow} 1.
\eeq
We now want to replace $p_{(i-j)/d}(u-w)$ by $p_{i/d}(u)$. First we replace $p_{(i-j)/d}(u-w)$ by $p_{(i-j)/d}(u)$. To this end, note that
\beq
|p_k(u) - p_k(u-w)| \le {\rm (cst)}\frac{|w|}{k^{\frac {1+d}2}}.
\eeq
We define
\beq
\tilde Q_{\ell,a}^{(4)}(u,v) = \tilde Q_{\ell,a}^{(5)}(u,v) - {\frak E}_{\ell,a}^{(5)}(u,v)
\eeq
with
\beq
\tilde Q_{\ell,a}^{(5)}(u,v) := 4\sum_{\substack{i \in [\gep_0,1-\gep_0]\ell\,\\ \ell -i \leftrightarrow v}}
\inftwo{1\le j\le a}{i-j \lra u}
p_{(i-j)/d}(u) p_{(\ell-i)/d}(v).
\eeq
and
\beq
{\frak E}_{\ell,a}^{(5)}(u,v) \le \frac{|K|}{\sqrt{\gep_0\ell -a}} \sum_{\substack{i \in [\gep_0,1-\gep_0]\ell\,\\ \ell -i \leftrightarrow v}} (i-a)^{-d/2} p_{(\ell - i)/d}(v).
\eeq
We can again deal with ${\frak E}_{\ell,a}^{(5)}(u,v)$ in a similar way as for the previous error terms. We now replace $p_{(i-j)/d}(u)$ by $p_{i/d}(u)$ in $\tilde Q_{\ell,a}^{(5)}(u,v)$. By computing the derivative of $f(s) := s^{-d/2} \exp(-|u|^2/(2s))$, we see that for all $\tilde A>0$:
\beq
\suptwo{1\le j \le a}{|u|\le \tilde An^{1/d}} |p_{i/d}(u) - p_{(i-j)/d}(u)|\le
 \frac{C(\gep,\gep_0,a,\tilde A)}{\ell} p_{i/d}(u).
\eeq
Note that the restriction of the integral over $z\in\bbR^d$ to $|u| \le \tilde An^{1/d}$ (that is $z\in\cB(x,\tilde A)$) can be made up to an error term not larger than $\gep_0$ provided $\tilde A$ is chosen large enough (uniformly in $\mu$, same argument as in Step 4).
Therefore,
\beq
\tilde Q_{\ell,a}^{(5)}(u,v) = \tilde Q_{\ell,a}^{(6)}(u,v) - {\frak E}_{\ell,a}^{(6)}(u,v)
\eeq
with
\beq
\tilde Q_{\ell,a}^{(6)}(u,v) := 4\sum_{\substack{i \in [\gep_0,1-\gep_0]\ell\,\\ \ell -i \leftrightarrow v}}p_{i/d}(u) p_{(\ell-i)/d}(v).
\eeq
and
\beq
{\frak E}_{\ell,a}^{(6)}(u,v) \le  \frac{C(\gep,\gep_0,a,\tilde A)}{\ell} \sum_{\substack{i \in [\gep_0,1-\gep_0]\ell\,\\ \ell -i \leftrightarrow v}}p_{i/d}(u) p_{(\ell - i)/d}(v).
\eeq
We are now left with showing that
\beq
4\int_{\cB(x,\tilde A)\cup \cB(y,\tilde A)} \dd z\ \frac1\gep n^{1-\frac2d} \sum_{\substack{i \in [\gep_0,1-\gep_0]\ell\,\\ \ell -i \leftrightarrow v}} \frac{p_{i/d}(u)p_{(\ell-i)/d}(v)}{\bP_u(S_\ell=v)}
\eeq
converges to
\beq
\int_{\cB(x,\tilde A)\cup \cB(y,\tilde A)} \dd z\ \frac1\gep \int_{\gep_0\gep}^{(1-\gep_0)\gep}\dd s \frac{p_{s/d}(x-z)p_{(\gep-s)/d}(z-y)}{p_{\gep/d}(x-y)},
\eeq
as $n\to\infty$, uniformly in $|x-y|\le A$, by a Riemann sum approximation. By the local limit theorem (Proposition~\ref{pr:LLT}),
\beq
\sup_{|x-y|\le A}
\Big|
\frac{2p_{\ell/d}(v-u)}{\bP_u(S_{\ell}=v)}
-1
\Big|
\le \frac{\rm (cst)}{\ell}
\eeq
so the proof will be complete once we prove that
\beq
2n^{1-\frac2d} \sum_{\substack{i \in [\gep_0,1-\gep_0]\ell\,\\ \ell -i \leftrightarrow v}} \frac{p_{i/d}(u)p_{(\ell-i)/d}(v)}{p_{\ell/d}(v-u)}
\longrightarrow
\int_{\gep_0\gep}^{(1-\gep_0)\gep}\dd s \frac{p_{s/d}(x-z)p_{(\gep-s)/d}(z-y)}{p_{\gep/d}(x-y)}
\eeq
as $n\to\infty$, uniformly on $z\in \cB(x,\tilde A)\cup \cB(y,\tilde A)$ and $|x-y|\le A$. Clearly,
\beq
n^{1-\frac2d}\sum_{\substack{i \in [\gep_0,1-\gep_0]\ell\,\\ \ell -i \leftrightarrow v}} \frac{p_{i/d}(u)p_{(\ell-i)/d}(v)}{p_{\ell/d}(v-u)}
=\sum_{\substack{i \in [\gep_0,1-\gep_0]\frac{\gep}{n^{2/d}}\,,\\ \ell - i\leftrightarrow v}} \frac{p_{in^{-2/d}/d}(\frac{u}{n^{-1/d}})p_{(\ell-i)n^{-2/d}/d}(\frac{v}{n^{-1/d}})}{p_{\gep n^{-2/d}/d}(\frac{v-u}{n^{-1/d}})} \frac{1}{n^{2/d}}
\eeq
and the rest is standard Riemann sum approximation together with the fact that the condition $\ell - i\leftrightarrow v$ reduces the number of terms by a factor $\tfrac12$ (recall~\eqref{eq:def-uv}).
\\

{\it Step 5. Conclusion.} By combining~\eqref{eq:LBCR1}, \eqref{eq:LBCR2}, \eqref{eq:LBCR3} and~\eqref{eq:LBCR4}, the proof is complete.

\begin{proof}[Proof of Lemma~\ref{lem:return-origin}]
By decomposing on the value of $H_0$ we obtain:
\beq
\ba
\bP_w(S_k=u, H_0 \le k) &=
\sum_{1\le r\le k} \bP_w(H_0=r)\bP_0(S_{k-r}=u)\\
&\le \sum_{1\le r \le k}\bP_0(S_r=w)\bP_0(S_{k-r}=u).
\ea
\eeq
By splitting the sum in two parts ($r\le k/2$ and $r>k/2$) and using the local limit theorem (Proposition~\ref{pr:LLT}) we get
\beq
\bP_w(S_k=u, H_0 \le k) \le {\rm (cst)} k^{-d/2}(\bE_0[\ell_\infty(w)]+\bE_0[\ell_\infty(u)]),
\eeq
where $\ell_\infty(\cdot):= \card\{i\ge 1\colon S_i = \cdot\}$. Since $|u|\wedge|w|\ge K$, a standard estimate yields
\beq
\bE_0[\ell_\infty(w)]\vee \bE_0[\ell_\infty(u)] \le {\rm (cst)} K^{2-d},
\eeq
which completes the proof.
\end{proof}

\bibliographystyle{abbrv}
\bibliography{SwissCheese}

\end{document}